\documentclass[10pt,a4paper]{article}
\usepackage[utf8]{inputenc}
\usepackage{amsmath}
\usepackage{pifont}
\usepackage{bm}
\usepackage[american]{babel}
\usepackage[all]{xy}
\usepackage{enumitem}
\usepackage{anysize}
\usepackage[dvipsnames]{xcolor}
\usepackage{graphicx}
\usepackage{stmaryrd}
\usepackage{amsfonts,amstext,amssymb,amsmath,amsthm,bigints}
\usepackage{fullpage}
\usepackage{moreverb}
\usepackage{pifont}
\usepackage[left=2cm,right=2cm,top=2cm,bottom=2cm]{geometry}
\usepackage{esint}
\usepackage{fourier-orns}
\usepackage{mathrsfs}
\usepackage{array}

\usepackage{hyperref}

\usepackage{authblk}
\newcolumntype{C}[1]{>{\centering\arraybackslash}b{#1}}
\newcolumntype{R}[1]{>{\raggedleft\arraybackslash}b{#1}}
\newcolumntype{L}[1]{>{\raggedright\arraybackslash}b{#1}}
\newcolumntype{M}[1]{>{\centering}m{#1}}
\newtheorem{theo}{Theorem}[section]
\newtheorem{defin}{Definition}[section]
\newtheorem{lem}{Lemma}[section]
\newtheorem{prop}{Proposition}[section]

\newtheorem{remark}{Remark}[section]

\newtheorem{cor}{Corollary}[section]
\usepackage{bbm}
\numberwithin{equation}{section}
\usepackage{subfigure}

\usepackage{pgfplots}
\usepackage{amsfonts,amsmath}
\usepackage{tikz-3dplot}
\pgfplotsset{compat=newest}
\usetikzlibrary{calc}
\usetikzlibrary{quotes,angles}
\usepackage{pgfplots}
\pgfplotsset{compat=1.12}
\usetikzlibrary{math}

\usepackage{mathtools,needspace}
\mathtoolsset{showonlyrefs}
\hypersetup{pdftitle={Space-time localized nonuniqueness for the inhomogeneous Navier-Stokes equations}}

\newcommand{\Ro}{\mathcal R_0}
\newcommand{\Rs}{\mathcal R}
\newcommand{\Bog}{\mathcal B}
\newcommand{\Sym}{\mathcal S}
\newcommand{\Pne}{\mathbb P_{\ne0}}
\newcommand{\TT}{\mathbb{T}}
\newcommand{\RR}{\mathbb{R}}
\newcommand{\NN}{\mathbb{N}}
\newcommand{\ZZ}{\mathbb{Z}}
\newcommand{\bfe}{\mathbf{e}}
\newcommand{\dd}{\,\mathrm{d}}
\newcommand{\HZ}{\mathcal H}
\DeclareMathOperator{\D}{div}
\DeclareMathOperator{\curl}{curl}
\DeclareMathOperator{\Id}{Id}
\DeclareMathOperator{\Tr}{Tr}
\DeclareMathOperator{\dist}{dist}
\DeclareMathOperator{\Supp}{supp}

\newcommand{\proofstep}[1]{\par\vspace{0.5em}\Needspace{4\baselineskip}\noindent\textbf{#1.}\par\nobreak\smallskip\nobreak}
\setlist[enumerate]{itemsep=3pt}
\setlist[itemize]{itemsep=3pt}

\date{}
\title{Space--time localized nonuniqueness for the \\ inhomogeneous Navier--Stokes equations}

\author{Xingyu Li}
\begin{document}
	\maketitle
	\begin{abstract}
		We prove space--time localized nonuniqueness for the inhomogeneous
		incompressible Navier--Stokes equations on the three-dimensional torus.
		For every $0<s<3/2$ and $p>1$ sufficiently close to one, we construct
		weak solutions with strictly positive density in $C_tL^p$ and velocity
		in $C_tH^s$. These solutions approximate prescribed smooth density
		and velocity fields satisfying the necessary mass and spatial first
		moment conditions, while agreeing with a given smooth solution outside
		a compact subset of a prescribed cylinder. In particular, there are
		infinitely many nontrivial solutions which coincide with the rest
		state outside that cylinder and whose density is not renormalized.

		The proof uses convex integration with concentrated density profiles
		transported exactly by the perturbed velocity. Nonnegative coefficients
		cancel the continuity defect and the stress simultaneously. A
		stationary density corrector preserves positivity through successive
		corrections, and
		compactly supported inverse divergence operators preserve spatial
		localization.
	\end{abstract}
	\tableofcontents
\section{Introduction}

The inhomogeneous incompressible Navier--Stokes equations describe a
viscous fluid whose density is transported by the flow. For densities
that are bounded and bounded away from zero the Cauchy problem is well
understood. Weak solutions with finite energy exist
\cite{Kaz74,Sim90,Lions96}, strong solutions are unique and exist
globally under smallness or two-dimensionality assumptions
\cite{LS75,Des97,Dan03,AGZ12,DM12,PZZ13,Li17,DM19}, and Hao, Shao, Wei
and Zhang~\cite{Hao25} recently proved uniqueness of Lions' weak
solutions in the two-dimensional whole space with $L^2$ initial
velocity, together with global well-posedness in three dimensions for
bounded initial density and small initial velocity in
$\dot H^{1/2}(\RR^3)$. All these results take place in a class of
solutions with bounded density and finite energy.

The question addressed here is how much of this uniqueness survives
when the density is only integrable to a power close to one and the
equations are imposed in the sense of distributions. We show that in
this regime uniqueness fails completely. For every $0<s<3/2$ there are
infinitely many weak solutions with strictly positive density in
$C_tL^p$, $p>1$ close to one, and velocity in $C_tH^s$, all emanating
from the same smooth initial data. Neither the positivity of the
density nor the viscosity obstructs the construction. Moreover, the
entire perturbation can be confined to an arbitrary space--time
cylinder: outside a compact subset of that cylinder the density,
velocity and pressure coincide with a given smooth solution.

The two regimes leave a gap. Uniqueness is known for bounded densities
in the energy class, and nonuniqueness is shown here for densities in
$L^p$ with $p$ close to one and without an energy inequality. Whether
uniqueness persists for densities in $L^q$ with large $q$, or for weak
solutions satisfying the energy inequality, remains open. For the
homogeneous equations, nonuniqueness below the energy class is known
from \cite{BV19,CL22,CL23}, and within the Leray class it is known in
the presence of a forcing term \cite{ABC22}; here the mechanism is the
density rather than a force.

We consider the equations on $\TT^3=\RR^3/\ZZ^3$,
\begin{equation}\label{eq:INS}
\left\{
\begin{aligned}
\partial_t\rho+\D(\rho u)&=0,\\
\partial_t(\rho u)+\D(\rho u\otimes u)+\nabla\Pi-\Delta u&=0,\\
\D u&=0.
\end{aligned}
\right.
\end{equation}
Here $\rho$ is a scalar density, $u$ an $\RR^3$-valued velocity,
and $\Pi$ a scalar pressure; the viscosity is normalized to one. We use the following notion of
weak solution, which does not impose an energy inequality.

\begin{defin}[Weak solutions]\label{def:weak-INS}
Let $0<s<3/2$ and $1<p<\infty$. A weak solution of \eqref{eq:INS}
consists of a density, velocity and pressure satisfying
\[
\rho\in C_tL^p,\qquad u\in C_tH^s,\qquad
\Pi\in\mathcal D'((0,1)\times\TT^3),
\]
with $\rho\ge\underline\rho$ almost everywhere for some constant
$\underline\rho>0$, and
\[
\rho u\in C_tL^1,\qquad \rho u\otimes u\in C_tL^1,
\]
such that \eqref{eq:INS} holds in the sense of distributions on
$(0,1)\times\TT^3$.
\end{defin}

\subsection{Main results}

Our first result approximates prescribed smooth density and velocity
fields inside an arbitrary cylinder.

Throughout, $B=B(x_0,r_B)\subset\TT^3$ denotes an open ball of radius
$r_B<1/4$. We identify $B$ with a Euclidean ball, so that moments such
as $\int_B(x-x_0)f\dd x$ of functions supported in $B$ are well defined.

\begin{theo}[Localized approximation]\label{thm:main}
For every $0<s<3/2$, there is $p_*(s)>1$ such that the following
holds for $1<p<p_*(s)$. Fix a cylinder
$Q\Subset(0,1)\times\TT^3$ with spatial base $B$, a ball of radius
less than $1/4$.

Let $(\bar\rho,\bar u,\bar\Pi)$ be a smooth solution of
\eqref{eq:INS} and $(\rho_*,u_*)$ a smooth pair on
$[0,1]\times\TT^3$, with $\bar\rho,\rho_*>0$ and $\D u_*=0$.
Assume
\begin{gather}
\Supp(\rho_*-\bar\rho,\;u_*-\bar u)\Subset Q,\notag\\
\int_B(\rho_*-\bar\rho)(t,x)\dd x=0,
\qquad
\int_B (x-x_0)(\rho_*-\bar\rho)(t,x)\dd x=0
\quad(0\le t\le1).\label{eq:moments}
\end{gather}
Then, for every $\delta>0$, there is a weak solution $(\rho,u,\Pi)$,
with $\Pi\in L^1_{t,x}$, satisfying
\[
\|\rho-\rho_*\|_{C_tL^p}+\|u-u_*\|_{C_tH^s}<\delta,
\qquad \rho\ge\rho_*-\delta\quad\text{almost everywhere}.
\]
Moreover, $(\bar\rho,\bar u,\bar\Pi)=(\rho,u,\Pi)$ outside a compact
subset of $Q$.
\end{theo}

The exponent $p_*(s)$ is explicit; see Section~\ref{sec:parameters}.
It tends to one as $s\uparrow3/2$, because the concentration needed to
make the velocity perturbation small in $H^s$ grows as $s$ approaches
$3/2$, and we have made no attempt to optimize it. The restriction
$r_B<1/4$ is only used to identify $B$ with a Euclidean ball.

The prescribed fields need not satisfy either evolution equation
inside $Q$. Applying the theorem to distinct smooth perturbations
gives the following consequence.

\begin{cor}[Compactly supported nonuniqueness]\label{cor:nonunique}
Let $s,p,(\bar\rho,\bar u,\bar\Pi)$ and $Q$ be as in
Theorem~\ref{thm:main}. There are infinitely many
distinct weak solutions which coincide with $(\bar\rho,\bar u,\bar\Pi)$
outside compact subsets of $Q$, converge to $(\bar\rho,\bar u)$ in
$C_tL^p\times C_tH^s$, and satisfy $\rho\ge\frac12\min\bar\rho$.
In particular, there are infinitely many weak solutions with
nonconstant density such that $\rho-1$, $u$ and $\Pi$ vanish outside a
compact subset of $Q$.
\end{cor}

\begin{remark}[Initial data and other domains]\label{rem:domains}
The solutions of Corollary~\ref{cor:nonunique} coincide with
$(\bar\rho,\bar u,\bar\Pi)$ for all times before the time interval of
$Q$. They are therefore infinitely many distinct weak solutions of the
Cauchy problem with the same smooth initial data
$(\bar\rho(0),\bar u(0))$; for the rest state these are $\rho_0=1$ and
$u_0=0$. Moreover, the equations are local and the perturbation is
supported in a compact subset of $Q$. Hence the same fields are weak
solutions on any domain $\Omega\supset B$, in particular on $\RR^3$
and on bounded domains with whatever boundary condition the reference
solution satisfies, and the nonuniqueness statement holds there
verbatim.
\end{remark}

\begin{remark}[Failure of renormalization]\label{rem:renormalization}
Let $(\rho,u,\Pi)$ be one of the solutions in the last assertion of
Corollary~\ref{cor:nonunique}, and let
$\Phi(z)=\sqrt{1+(z-1)^2}-1$. Then $\Phi\ge0$, $\Phi(z)=0$ only for
$z=1$, $|\Phi'|\le1$ and $\Phi(z)\le|z-1|$. Hence
$\Phi(\rho)\in C_tL^1$ and
$|\Phi(\rho)u|\le(\rho+1)|u|\in L^1_{t,x}$, since $\rho u\in C_tL^1$
and $u\in C_tL^2$. If the renormalized equation
$\partial_t\Phi(\rho)+\D(\Phi(\rho)u)=0$ held in
$\mathcal D'((0,1)\times\TT^3)$, testing with functions of time only
would show that $t\mapsto\int\Phi(\rho(t))\dd x$ is constant. It
vanishes for $t$ before the time interval of $Q$, where $\rho=1$.
Hence $\Phi(\rho(t))=0$ and $\rho(t)=1$ for every $t$, contradicting
the nonconstancy of $\rho$.
This is consistent with the DiPerna--Lions theory \cite{DPL89,Amb04},
which guarantees renormalization for $\rho\in L^p$ and
$u\in W^{1,p'}$ with $1/p+1/p'=1$. Here $u$ has at most one derivative
in $L^2$, far from the dual exponent $p'$ required when $p$ is close to
one, and the pair of exponents lies in the nonuniqueness range for the
transport equation found by Modena and
Sz\'ekelyhidi~\cite{MSz18,MSz19}. The point of the present result is
that positive density, the momentum equation and viscosity are
compatible with this phenomenon.
\end{remark}

\begin{remark}[The spatial moment conditions]\label{rem:necessity}
Conditions \eqref{eq:moments} from Theorem~\ref{thm:main} cannot be
dropped. Let $(\rho,u,\Pi)$ be any weak solution with
$(\rho,u,\Pi)=(\bar\rho,\bar u,\bar\Pi)$ outside a compact set
$K\subset Q$, and let $\chi_B\in C^\infty_c(B)$ equal one on the
spatial projection of $K$. Test the continuity equations of both
solutions with $\psi(t)\chi_B$ and with $\psi(t)(x-x_0)_i\chi_B$, and
the momentum equations with $\psi(t)\bfe_i\chi_B$, where
$\psi\in C^\infty_c((0,1))$. All terms containing $\nabla\chi_B$
cancel, because the two solutions agree there. In
$\mathcal D'((0,1))$ this gives
\[
\begin{aligned}
\frac{\mathrm d}{\mathrm dt}\int(\rho-\bar\rho)\chi_B&=0,\\
\frac{\mathrm d}{\mathrm dt}\int(x-x_0)(\rho-\bar\rho)\chi_B
&=\int(\rho u-\bar\rho\bar u)\chi_B,\\
\frac{\mathrm d}{\mathrm dt}\int(\rho u-\bar\rho\bar u)\chi_B&=0.
\end{aligned}
\]
The three quantities are continuous in time and vanish for $t$ before
the time interval of $Q$. Hence the mass and the first moment of
$\rho-\bar\rho$ in $B$ vanish for every $t$. These two quantities are
continuous with respect to $C_tL^1$ convergence of $\rho$. A smooth pair
$(\rho_*,u_*)$ violating \eqref{eq:moments} at some time is therefore
at positive $C_tL^1$ distance from the densities of all such solutions,
and cannot be approximated as in Theorem~\ref{thm:main}.
\end{remark}

\begin{remark}[Localization in time]\label{rem:time-only}
If $B$ is replaced by $\TT^3$, the same scheme works with the periodic
operators $\Ro$, $\Rs$ in place of the compactly supported ones of
Section~\ref{sec:compact}. The correction $\theta_c$ then only removes
the spatial mean of the density increment, and no angular momentum
correction is needed. This gives approximation of prescribed fields
which coincide with different exact solutions before and after the
time interval, provided these solutions have the same mass and the same
momentum; in particular, two smooth solutions with equal total mass and
momentum can be glued across a prescribed time interval. Since the
argument is a simplification of the one given here, we do not repeat
it.
\end{remark}

\subsection{Main ideas of the proof}

We construct smooth approximate solutions with a continuity defect
$R$ and a symmetric stress $T$:
\begin{equation}\label{eq:intro-approximate}
\left\{
\begin{aligned}
\partial_t\rho+\D(\rho u)&=\D R,\\
\partial_t(\rho u-R)+\D(\rho u\otimes u)+\nabla\Pi-\Delta u&=\D T,\\
\D u&=0.
\end{aligned}
\right.
\end{equation}
At each step we reduce $R,T$, keeping the perturbation compactly
supported and the density bounded below. The term $-R$ in the momentum
variable is important: it lets the same density perturbation correct
both equations.

\subsubsection*{The principal perturbations}

Let $\Theta_k\ge0$ have unit spatial mean and move with velocity
$V_k=\omega k$. We choose their supports disjoint and construct
a divergence-free perturbation $w$ satisfying
\[
\partial_t\Theta_k+V_k\cdot\nabla\Theta_k=0,
\qquad
u+w=V_k\quad\text{on }\Supp(a_k\Theta_k).
\]
The principal density perturbation is
\[
\theta_p=\sum_k a_k(\Theta_k-1),\qquad a_k\ge0.
\]
The transport of each profile is exact, so its remaining error contains
only derivatives of the coefficient:
\[
\begin{aligned}
&\partial_t[a_k(\Theta_k-1)]
+\D[a_k(\Theta_k-1)V_k]\\
&\qquad=(\Theta_k-1)(\partial_t+V_k\cdot\nabla)a_k.
\end{aligned}
\]
This is the reason for moving the profiles. A stationary bump would
leave the full time derivative of $a_k\Theta_k$ as an error, whereas a
bump moving with the constant velocity $V_k$ is transported exactly
once $u+w=V_k$ on its support.

We choose the coefficients to satisfy
\begin{equation}\label{eq:intro-cancellation}
\begin{aligned}
\sum_k a_k(V_k-u)&=-R,\\
\sum_k a_k(V_k\otimes V_k-u\otimes u)&=\pi\Id-T.
\end{aligned}
\end{equation}
The first identity cancels the continuity defect; the second cancels
the stress after a pressure correction. Opposite coordinate directions
give the first identity, while the geometric lemma supplies the second.
Taking $\omega$ large makes the nonnegative coefficients small.

\subsubsection*{Concentration and the regularity range}

At frequency $\lambda$ and concentration $\mu$, the profiles and
velocity perturbation satisfy
\[
\|\Theta_k\|_{L^p}\lesssim\mu^{3-3/p},\qquad
\|w\|_{H^s}\lesssim\omega\lambda^s\mu^{s-3/2}.
\]
Choosing $\mu=\lambda^N$, with
$N(3/2-s)>s+2$, makes the velocity perturbation small for
$s<3/2$. Concentration increases the density norm, but
$3-3/p\to0$ as $p\downarrow1$. Thus $p$ can be chosen so that the
small coefficients in $\theta_p$ compensate for this increase. This
explains why the admissible range of $p$ depends on $s$.

\subsubsection*{The density lower bound}

Proposition~\ref{prop:main} requires a positive density lower bound
throughout the iteration. The principal perturbation can decrease
the density wherever its positive bumps vanish. With
$\omega=\delta_n^{-4}$, the coefficient bound gives
\[
\theta_p\ge-\sum_k a_k,\qquad
\sum_k a_k\le\delta_n^2|R|+\delta_n^4|T|+C\delta_n^{12}.
\]
Here $\delta_n$ is the small parameter at step $n$; the weights
$\delta_n^2,\delta_n^4$ bound the possible density loss when
correcting $R,T$.

At the next stage the same subtraction is determined by $R^+,T^+$.
In particular, $T^+$ contains $-(\nabla w+\nabla w^T)$, which can
be large on the velocity cutoff regions where the density bumps
vanish. Small $L^p$ norms of the new defects therefore do not give
a small pointwise density loss at the following stage.

We add a stationary density corrector $\theta_d$ along the lines
traversed by the moving bumps. Its profile has bounded $L^p$ norm
and dominates these new errors pointwise. For the total perturbation
$\theta=\theta_p+\theta_c+\theta_d$, including the small spatial
moment correction $\theta_c$, we prove
\begin{equation}\label{eq:intro-positivity}
\theta\ge-\sum_k a_k
+\delta_{n+1}^2|R^+|+\delta_{n+1}^4|T^+|
-C\delta_{n+1}^{16}.
\end{equation}
The positive defect terms cover the subtraction when $R^+,T^+$
are corrected at the next step. The remaining losses are summable,
giving the density lower bound in the iteration proposition.

The profile is stationary so that its time derivative falls only on
the time cutoff. The resulting error $\Ro\partial_t\theta_d$ also
has to be covered. We use successive spatial cutoffs with decreasing
coefficients: each supplies density where the preceding inverse
divergence extends beyond its source.

\subsubsection*{Compact support}

For a symmetric solution of $\D S=f$ compactly supported in $B$,
the necessary compatibility conditions are
\[
\int f=0,\qquad \int x\times f=0.
\]
These are the vanishing linear and angular momentum conditions used
in localized convex integration; see Isett--Oh~\cite{IO16} and Zhao~\cite{Zhao}.
We arrange them before applying the inverse divergence operators.
The density increment $\theta$ is corrected so that
\[
\int\theta=0,\qquad\int x\theta=0.
\]
Writing $m=\rho u-R$, the modified momentum increment satisfies
$\partial_t\theta+\D(m^+-m)=0$, hence
\[
\int(m^+-m)=-\int x\D(m^+-m)
=\partial_t\int x\theta=0.
\]
A compactly supported curl removes its angular momentum without
changing this divergence. The resulting momentum residual then has
a compactly supported symmetric inverse divergence.

Finally, the approximate equations are linear in
$\rho$, $u$, $\Pi$, $\rho u-R$ and $\rho u\otimes u-T$.
Mollifying their differences from the fixed approximate solution
therefore preserves the equations and only slightly enlarges the
compact support. This provides the derivative bounds needed to
continue the iteration.

\subsection{Related work}

Convex integration for fluid equations goes back to the
differential-inclusion formulation of De Lellis and
Sz\'ekelyhidi~\cite{DLS09}, which builds on the work of Nash, Kuiper
and Gromov~\cite{Nash54,Kuiper55,Gromov86}; see the
surveys~\cite{DLS17,BV19survey}. Continuous dissipative Euler
flows~\cite{DLS13}, the Mikado flows of Daneri and
Sz\'ekelyhidi~\cite{DS17} and Isett's proof of Onsager's
conjecture~\cite{Isett18} established the geometric framework used
below. Compact support in time was obtained by Isett~\cite{Isett17},
and compact support in space, subject to the linear and angular
momentum conditions that reappear in Section~\ref{sec:compact}, by
Isett and Oh~\cite{IO16}, see also ~\cite{Zhao}.

For the Navier--Stokes equations, Buckmaster and Vicol~\cite{BV19}
introduced intermittent building blocks whose concentration overcomes
viscosity. Sharp Sobolev nonuniqueness and temporal intermittency
were obtained by Cheskidov and Luo~\cite{CL22,CL23}, and partial
regularity in time in~\cite{BCV22}. These solutions lie outside the
Leray class; within it, nonuniqueness is known for the forced
equations~\cite{ABC22}. Our solutions have velocity in $H^s$ for any
prescribed $s<3/2$, but no energy inequality is imposed, and the
nonuniqueness is produced by the density.

For the transport equation with Sobolev velocity, nonuniqueness was
established by Modena and Sz\'ekelyhidi~\cite{MSz18,MSz19} in the
range complementary to the DiPerna--Lions theory~\cite{DPL89,Amb04}.
Moving profiles with full dimensional concentration appear
in~\cite{MS20}, positive profiles and their characteristic geometry
in~\cite{BCDL21}, critical constructions in~\cite{CL21}, and temporal
intermittency in~\cite{CL24}. Convex integration for the inhomogeneous
incompressible Euler equations was carried out by Giri and
Koley~\cite{GK24}.

Here the density and velocity are constructed together. The identities
\eqref{eq:intro-cancellation} cancel the continuity defect and the stress
using the same nonnegative coefficients, and the stationary density
corrector gives \eqref{eq:intro-positivity}. For spatial localization,
we use the Bogovskii operator~\cite{Bog79,Galdi} and solve the symmetric
divergence equation subject to the linear and angular momentum
conditions, as in~\cite{IO16}.

\subsection{Organization of the paper}

Section~\ref{sec:parameters} formulates the iteration and deduces the
main results. Section~\ref{sec:tools} collects the auxiliary estimates
and constructs the stationary profile. We construct one step in
Section~\ref{sec:construction} and prove its estimates in
Section~\ref{sec:estimates}. Section~\ref{sec:induction} completes the
iteration. Only the elementary cutoff, calculus, and kernel arguments are
collected in Appendix~\ref{app:elementary}.

\subsection{Notation}

All spatial norms and integrals are over $\TT^3$ unless a domain is
specified. We write $C_tX=C([0,1];X)$, and spatial norms of
time-dependent functions are taken uniformly in time. For a nonnegative integer $m$, we set
\[
\|f\|_{C^m_{t,x}}
=\max_{j+|\alpha|\le m}\|\partial_t^j\partial_x^\alpha f\|_\infty.
\]
We use the same notation for scalar, vector, and matrix fields, with
the Euclidean norm in the range; a norm of a tuple is the sum of the
norms of its entries. The letter $p$ denotes the density exponent,
and $q$ is used for auxiliary exponents.

We denote the coordinate unit vectors by $\bfe_1,\bfe_2,\bfe_3$
and the space of symmetric $3\times3$ real matrices by
$\mathrm{Sym}_3$. For vectors and matrices we use the conventions
\[
(v\otimes w)_{ij}=v_iw_j,\qquad
(\nabla v)_{ij}=\partial_jv_i,\qquad
(\D S)_i=\sum_j\partial_jS_{ij}.
\]
We also write $(S\nabla\chi)_i=\sum_jS_{ij}\partial_j\chi$.
In component calculations, repeated indices are summed from one to
three, and $\varepsilon_{ijk}$ denotes the alternating symbol.
The operators $\nabla$, $\Delta$, $\D$ and $\curl$ act only in space,
and $|\cdot|$ is the Euclidean norm, including for matrices. We set
\[
\fint f=\int_{\TT^3}f\dd x,\qquad \Pne f=f-\fint f,
\]
and inverse Laplacians act on mean-zero functions.

\subsection*{Acknowledgements}
XL would like to thank Xiaoyutao Luo for fruitful discussions and for his invitation to visit the Academy of Mathematics and Systems Science, Chinese Academy of Sciences.  XL is supported by PRIN 2022 ``Turbulent
effects vs Stability in Equations from Oceanography'' (TESEO), project
number 2022HSSYPN. 
\section*{Declaration on the use of AI tools}
During the preparation of this work the author used Chatgpt 6 and Claude as research assistants. The strategy for spatial localization in
Section~\ref{sec:compact} was organized by Claude on the basis of the ideas in \cite{Zhao}. The models were also used to check
algebraic identities and to improve the exposition.

All mathematical arguments and conclusions in that lemma were independently
checked and verified by the author, who assumes full responsibility for its content.

\section{The iteration and proof of the main results}
\label{sec:parameters}

We first extend the prescribed fields to a smooth approximate solution
with compactly supported errors. Proposition~\ref{prop:main} reduces
these errors to zero with summable increments, while preserving
positivity and a common compact support. We then deduce the main
results from this proposition.

\subsection{The approximate system}

We introduce defects in both evolution equations. The continuity
defect is included in the momentum variable, so that both equations
remain linear in the density, modified momentum and modified flux.
This form will be used both for the principal cancellation and for
regularization.

\begin{defin}[Approximate solutions]\label{def:system}
A smooth solution of the approximate system is a quintuple
$Z=(\rho,u,\Pi,R,T)$ on $[0,1]\times\TT^3$, where $\rho,\Pi$ are
smooth scalar functions, $u,R$ smooth $\RR^3$-valued fields, and $T$ a
smooth $\mathrm{Sym}_3$-valued field, satisfying
\begin{equation}\label{eq:defect}
\left\{
\begin{aligned}
\partial_t\rho+\D(\rho u)&=\D R,\\
\partial_t(\rho u-R)+\D(\rho u\otimes u)+\nabla\Pi-\Delta u&=\D T,\\
\D u&=0.
\end{aligned}
\right.
\end{equation}

Define the smooth vector field $m$ and symmetric matrix field $\Phi$ by
\begin{equation}\label{eq:modified-variables}
\begin{aligned}
m(t,x)&=\rho(t,x)u(t,x)-R(t,x),\\
\Phi(t,x)&=\rho(t,x)u(t,x)\otimes u(t,x)-T(t,x).
\end{aligned}
\end{equation}
Then \eqref{eq:defect} becomes
\begin{equation}\label{eq:linear-defect}
\left\{
\begin{aligned}
\partial_t\rho+\D m&=0,\\
\partial_tm+\D\Phi+\nabla\Pi-\Delta u&=0,\\
\D u&=0.
\end{aligned}
\right.
\end{equation}
\end{defin}

\subsection{Parameters}

Given $0<s<3/2$, fix a small exponent $0<\gamma<10^{-3}$ and a
sufficiently large integer $N$ such that
\[
N\ge\gamma^{-1},\qquad N(3/2-s)>s+2.
\]
Set $b=N^8$ and take the density exponent in the range
\[
1<p<p_*(s),\qquad p_*(s)=1+\frac{\gamma}{N^{10}}.
\]

For an integer $a\ge2$, define the frequencies and error sizes by
\begin{equation}\label{eq:iteration-scales}
\lambda_n=a^{b^n},\qquad \delta_n=\lambda_n^{-\gamma}
\qquad(n\ge0).
\end{equation}
The exponents $\gamma,N,b,p$ remain fixed. We take $a>0$ sufficiently
large, depending on the prescribed fields, and enlarge only $a$ in the
proof.

\subsection{Initial defects}

Write $Q=J\times B$ and $B=B(x_0,r_B)$, where $J\Subset(0,1)$ is
an open interval and $r_B<1/4$. We first place the prescribed pair in
the approximate system. The moment assumptions allow us to choose
both defects compactly supported.

\begin{lem}[Local extension]\label{lem:ext}
Under the hypotheses of Theorem~\ref{thm:main} there are smooth fields
$R_*$ and $T_*=T_*^T$, supported in a compact set $K_*\subset Q$, such
that $Z_*=(\rho_*,u_*,\bar\Pi,R_*,T_*)$ solves \eqref{eq:defect} from Definition~\ref{def:system}, and
$Z_*=(\bar\rho,\bar u,\bar\Pi,0,0)$ outside $K_*$.
\end{lem}

\begin{proof}
Let $K_0\subset Q$ be compact with $(\rho_*,u_*)=(\bar\rho,\bar u)$
outside $K_0$, let $J_0\Subset J$ be a compact interval and $r'<r_B$ a
radius with $K_0\subset J_0\times B(x_0,r')$, and put $U'=B(x_0,r')$.
Fix a scalar $\phi_\star\in C^\infty_c(U')$ with
$\int\phi_\star=1$.

\proofstep{Step 1: The continuity defect}
Since $\bar\rho$ solves the continuity equation,
\[
f:=\partial_t\rho_*+\D(\rho_*u_*)
 =\partial_t(\rho_*-\bar\rho)+\D(\rho_*u_*-\bar\rho\bar u)
\]
is supported in $K_0$, and $\int f\dd x=\frac{\mathrm d}{\mathrm dt}\int_B(\rho_*-\bar\rho)\dd x=0$
by Theorem~\ref{thm:main}. Lemma~\ref{lem:bog} gives the vector field
$R^{\mathrm b}:=\Bog_{U'}f\in C^\infty_c(U';\RR^3)$ with
$\D R^{\mathrm b}=f$, smooth in $t$ and vanishing for $t\notin J_0$.
Put
\[
\begin{aligned}
F^0&:=\rho_*u_*-\bar\rho\bar u-R^{\mathrm b},\\
D_*(t,x)&:=\tfrac12\curl_x\left[
\phi_\star(x)\int(y-x_0)\times F^0(t,y)\dd y\right],\\
F_*&:=F^0-D_*.
\end{aligned}
\]
These fields are smooth and supported in $J_0\times U'$. Moreover
$\D F_*=\D(\rho_*u_*-\bar\rho\bar u)-f=-\partial_t(\rho_*-\bar\rho)$.
By Lemma~\ref{lem:moment-identities} and Theorem~\ref{thm:main},
\[
\begin{aligned}
\int F_*\dd x
&=-\int(x-x_0)\D F_*\dd x\\
&=\frac{\mathrm d}{\mathrm dt}
\int(x-x_0)(\rho_*-\bar\rho)\dd x=0.
\end{aligned}
\]
The curl correction also gives
\[
\int(x-x_0)\times F_*\dd x
=\left(1-\int\phi_\star\dd x\right)
\int(x-x_0)\times F^0\dd x=0.
\]

\proofstep{Step 2: The momentum defect}
Let
\[
g:=\partial_tF_*+\D\bigl(\rho_*u_*\otimes u_*-\bar\rho\bar u\otimes\bar u\bigr)
 -\Delta(u_*-\bar u).
\]
It is smooth and supported in $J_0\times U'$. Since
$\Delta v=\D(\nabla v+\nabla v^T)$ for divergence-free $v$,
Lemma~\ref{lem:moment-identities} gives
\[
\int g=\partial_t\int F_*=0,\qquad
\int(x-x_0)\times g=\partial_t\int(x-x_0)\times F_*=0.
\]
Set
\[
R_*:=R^{\mathrm b}+D_*,\qquad T_*:=\Sym_{U'}g,
\]
with $\Sym_{U'}$ from Lemma~\ref{lem:sym}. Then $\D R_*=f$, which is the
continuity equation in \eqref{eq:defect}. Also
$\rho_*u_*-R_*=\bar\rho\bar u+F_*$, and the momentum equation for
$(\bar\rho,\bar u,\bar\Pi)$ gives
\[
\partial_t(\rho_*u_*-R_*)+\D(\rho_*u_*\otimes u_*)+\nabla\bar\Pi-\Delta u_*
=g=\D T_*.
\]
The fields $R_*,T_*$ are supported in
$K_*:=J_0\times\overline{U'}\subset Q$, and the prescribed fields
agree with the reference solution outside this set.
\end{proof}

From now on we fix $Z_*$ and $K_*$ as in Lemma~\ref{lem:ext}, and a
radius $r_0\in[r_B/2,r_B)$ with $K_*\subset J\times B(x_0,r_0)$.
Constants called $C_*$ may depend on $Z_*$.

\subsection{The iteration}

The iteration must do more than make the two defects vanish. It must
also preserve a common compact support, a positive lower bound for
the density, and strong convergence of the nonlinear terms. The next
proposition records precisely these conclusions.

\begin{prop}[Iteration]\label{prop:main}
Let $Z_*$ and $K_*$ be as in Lemma~\ref{lem:ext}. For sufficiently
large $a>0$, there are a compact set $K_\infty\subset Q$ containing
$K_*$ and smooth solutions
\[
Z_n=(\rho_n,u_n,\Pi_n,R_n,T_n),\qquad n\ge0,
\]
of \eqref{eq:defect} from Definition~\ref{def:system}, with $Z_0=Z_*$. For every $n$, both $Z_n-Z_*$
and $(R_n,T_n)$ are supported in $K_\infty$.

The density and velocity increments satisfy
\begin{equation}\label{eq:iteration-increments}
\|\rho_{n+1}-\rho_n\|_{C_tL^p}
+\|u_{n+1}-u_n\|_{C_tH^s}\le\delta_n,
\end{equation}
and the defects tend to zero according to
\begin{equation}\label{eq:iteration-defects}
\|(R_n,T_n)\|_{C_tL^p}\le\delta_n^8\qquad(n\ge1).
\end{equation}
The nonlinear quantities and pressure have summable increments:
for $n\ge1$,
\begin{equation}\label{eq:iteration-products}
\begin{aligned}
&\|\rho_{n+1}u_{n+1}-\rho_nu_n\|_{C_tL^1}\\
&\quad+\|\rho_{n+1}u_{n+1}\otimes u_{n+1}
-\rho_nu_n\otimes u_n\|_{C_tL^1}\\
&\quad+\|\Pi_{n+1}-\Pi_n\|_{L^1_{t,x}}\le\delta_n.
\end{aligned}
\end{equation}
Finally, the density retains the pointwise lower bound
\begin{equation}\label{eq:iteration-positivity}
\rho_n\ge\rho_*-\delta_0.
\end{equation}
\end{prop}

The proposition is proved in Section~\ref{sec:induction}, using the
one-step Proposition~\ref{prop:step}.

\subsection{Proof of Theorem~\ref{thm:main}}

\begin{proof}[Proof of Theorem~\ref{thm:main}]
Let $Z_*$ be given by Lemma~\ref{lem:ext}, and apply
Proposition~\ref{prop:main} with $a$ so large that
$2\delta_0<\min\{\delta,\frac12\min\rho_*\}$.
By \eqref{eq:iteration-increments}, $\rho_n\to\rho$ in $C_tL^p$ and $u_n\to u$ in $C_tH^s$, with
\[
\|\rho-\rho_*\|_{C_tL^p}+\|u-u_*\|_{C_tH^s}\le\sum_{n\ge0}\delta_n
\le2\delta_0<\delta.
\]
By \eqref{eq:iteration-positivity}, $\rho\ge\rho_*-\delta_0\ge\rho_*-\delta$; in particular
$\rho_n,\rho\ge\frac34\min\rho_*$. The increments from $Z_0$ to $Z_1$ are finite by smoothness, and
\eqref{eq:iteration-products} makes all later increments summable.
Hence $\rho_nu_n$ and $\rho_nu_n\otimes u_n$ converge in $C_tL^1$ to
limits $m$ and $\Phi$, and $\Pi_n\to\Pi$ in $L^1_{t,x}$. By \eqref{eq:iteration-defects},
$R_n,T_n\to0$ in $C_tL^1$.

Fix $t$. Along a subsequence, $\rho_n(t)\to\rho(t)$, $u_n(t)\to u(t)$,
$\rho_nu_n(t)\to m(t)$ and $\rho_nu_n\otimes u_n(t)\to\Phi(t)$ almost
everywhere. Hence $m(t)=\rho u(t)$ and $\Phi(t)=\rho u\otimes u(t)$.
Since $t$ was arbitrary, $\rho u\in C_tL^1$, $\rho u\otimes u\in C_tL^1$,
and the convergence holds in $C_tL^1$.

For $\phi\in C^\infty_c((0,1)\times\TT^3)$ and
$\psi\in C^\infty_c((0,1)\times\TT^3;\RR^3)$, \eqref{eq:defect} gives
\begin{gather*}
\iint\rho_n\partial_t\phi+(\rho_nu_n-R_n)\cdot\nabla\phi\dd x\dd t=0,\\
\iint(\rho_nu_n-R_n)\cdot\partial_t\psi
 +(\rho_nu_n\otimes u_n-T_n):\nabla\psi
 +\Pi_n\D\psi+u_n\cdot\Delta\psi\dd x\dd t=0.
\end{gather*}
All terms converge, so $(\rho,u,\Pi)$ satisfies
Definition~\ref{def:weak-INS}. By the support property in
Proposition~\ref{prop:main} and Lemma~\ref{lem:ext}, $(\rho_n,u_n,\Pi_n)=(\bar\rho,\bar u,\bar\Pi)$
outside $K_\infty$ for every $n$; hence the same holds for the limit.
\end{proof}

\subsection{Proof of Corollary~\ref{cor:nonunique}}

\begin{proof}
We distinguish the solutions by a linear observable of the density.
Choose $t_*\in J$, $\eta_0\in C^\infty_c(J;[0,1])$ with
$\eta_0(t_*)=1$, and $g\in C^\infty_c(B)$ with
$\Delta g\not\equiv0$. Put
$f=\Delta g/\|\Delta g\|_\infty$. Then
$\| f\|_\infty=1$, and integration by parts gives
$\int f=0$ and $\int(x-x_0) f=0$. Let
$c_f=\| f\|_{L^2}^2\in(0,1)$ and
$\alpha_j=2^{-j-4}\min\{1,\min\bar\rho\}$ for $j\ge1$. The pairs
\[
\rho_{*,j}=\bar\rho+\alpha_j\eta_0(t) f(x),\qquad u_{*,j}=\bar u
\]
satisfy the hypotheses of Theorem~\ref{thm:main}. Apply it with
$\delta_j=\alpha_jc_f/4$. Since the $L^1$ norm on $\TT^3$ is
bounded by the $L^p$ norm,
\[
\int\bigl(\rho_j(t_*)-\bar\rho(t_*)\bigr) f\dd x
\in\bigl[\alpha_jc_f-\delta_j,\ \alpha_jc_f+\delta_j\bigr]
=\alpha_jc_f\bigl[\tfrac34,\tfrac54\bigr].
\]
These intervals are pairwise disjoint because $\alpha_{j+1}=\alpha_j/2$,
so the solutions are distinct. Moreover
$\rho_j\ge\bar\rho-\alpha_j-\delta_j\ge\frac12\min\bar\rho$ and
$\|\rho_j-\bar\rho\|_{C_tL^p}+\|u_j-\bar u\|_{C_tH^s}\le\alpha_j+\delta_j\to0$.
For $(\bar\rho,\bar u,\bar\Pi)=(1,0,0)$ the observable is nonzero, so
$\rho_j$ is not constant, while $\rho_j-1$, $u_j$ and $\Pi_j$ vanish
outside a compact subset of $Q$.
\end{proof}

\section{Preliminaries}\label{sec:tools}

We collect the inverse divergence estimates, the geometric lemma, and
the stationary density profile used in the perturbation. The periodic
operators give a gain in frequency; their localized versions preserve
compact support under the appropriate moment conditions. Constants
depend only on the fixed parameters and the quantities indicated in
each statement.

\subsection{Time cutoffs and calculus inequalities}

We use three nested time cutoffs. The first supports the density
weights, the second the velocity perturbation, and the third the
stationary density correction. Their separation makes the last cutoff
equal to one wherever the principal errors are present.

\begin{lem}[Time cutoffs]\label{lem:cutoffs}
Let $I=[t_-,t_+]\Subset J'\subset(0,1)$ be intervals, $J'$ open, and
$\tau=\dist(I,\partial J')$. There are
$\eta,\eta_u,\eta_d\in C_c^\infty(J';[0,1])$ such that
$\eta=1$ near $I$, $\eta_u=1$ near $\Supp\eta$, and $\eta_d=1$ near
$\Supp\eta_u$. For every integer $j\ge0$,
\begin{equation}\label{eq:time-cutoff-bounds}
\begin{aligned}
\|\eta\|_{C^j}+\|\eta_u\|_{C^j}+\|\eta_d\|_{C^j}&\le C_j\tau^{-j},\\
|\eta_d'|&\le C_N\tau^{-1}\eta_d^{1-1/(16N)}.
\end{aligned}
\end{equation}
\end{lem}

The proof is given in Appendix~\ref{app:elementary}.

\begin{lem}[Calculus inequalities]\label{lem:calculus}
Let $j\ge0$ be an integer. For smooth scalar functions $f,g$ on
$[0,1]\times\TT^3$,
\[
\|fg\|_{C^j_{t,x}}\le2^j\|f\|_{C^j_{t,x}}\|g\|_{C^j_{t,x}}.
\]
Let $\mathcal K\subset\RR^m$ be compact, let $\Gamma$ be a smooth
scalar function on a neighbourhood of $\mathcal K$, and let
$v:[0,1]\times\TT^3\to\mathcal K$ be smooth. Then
\[
\|\Gamma\circ v\|_{C^j_{t,x}}\le C\bigl(j,m,\|\Gamma\|_{C^j(\mathcal K)}\bigr)
\bigl(1+\|v\|_{C^j_{t,x}}\bigr)^j.
\]
If $\Gamma$ is smooth on all of $\RR^m$ and all its derivatives of
order $1,\dots,j$ are bounded, the same bound holds for the derivatives
of order $1,\dots,j$ of $\Gamma\circ v$, with $\|\Gamma\|_{C^j(\mathcal K)}$
replaced by $\max_{1\le l\le j}\|\nabla^l\Gamma\|_{L^\infty(\RR^m)}$.
The estimates apply componentwise to vector and matrix fields.
\end{lem}

The proof is given in Appendix~\ref{app:elementary}.

\subsection{Periodic operators}

The periodic operators give a gain in frequency on mean-zero
profiles. We later localize them with a cutoff and a Bogovskii
corrector.

\begin{lem}[Periodic antidivergence operators]\label{lem:operators}
For $f\in C^\infty(\TT^3)$ and
$F\in C^\infty(\TT^3;\RR^3)$ define
\begin{equation}\label{eq:inverse-operators}
\Ro f=\nabla\Delta^{-1}\Pne f,\qquad
\Rs F=\nabla\Delta^{-1}\Pne F+(\nabla\Delta^{-1}\Pne F)^T
 -\nabla^2\Delta^{-2}\D F.
\end{equation}
Then $\D\Ro f=\Pne f$, $\Rs F$ is symmetric, and
$\D\Rs F=\Pne F+\nabla\Delta^{-1}\D F-\nabla\Delta^{-1}\D F=\Pne F$.

Let $G\in\mathcal D'(\TT^3)$ be the mean-zero solution of
$\Delta G=\delta_0-1$, where $\delta_0$ is the Dirac mass at the
origin, and let $G_2=G*G$. These periodic distributions are smooth
away from the origin. For every multi-index $\alpha$ and
$z\in[-\tfrac12,\tfrac12]^3\setminus\{0\}$,
\begin{equation}\label{eq:kernel}
\begin{aligned}
|\partial^\alpha G(z)|&\le C_\alpha(1+|z|^{-1-|\alpha|}),\\
|\partial^\alpha G_2(z)|&\le C_\alpha(1+|z|^{1-|\alpha|}).
\end{aligned}
\end{equation}

In particular, $\Ro$, $\Rs$ and
$\Delta^{-1}\Pne$ have zero-mean convolution kernels $K$ satisfying,
for every integer $j\ge0$,
\begin{equation}\label{eq:operator-kernel}
|\nabla^jK(z)|\le C_j\bigl(1+\dist(z,0)^{-2-j}\bigr).
\end{equation}
They are bounded on $L^q$ for $1\le q\le\infty$ and on $C^j_{t,x}$,
and they commute with derivatives and translations.

For every integer $\lambda\ge1$, the scaling identities are
\begin{equation}\label{eq:operator-scaling}
\begin{aligned}
\mathcal L[g(\lambda\cdot)]
&=\lambda^{-1}(\mathcal Lg)(\lambda\cdot),
&&\mathcal L\in\{\Ro,\Rs\},\\
\Delta^{-1}\Pne[g(\lambda\cdot)]
&=\lambda^{-2}(\Delta^{-1}\Pne g)(\lambda\cdot).
\end{aligned}
\end{equation}
Finally,
\begin{equation}\label{eq:operator-integrable-derivatives}
\|\Delta^{-2}\partial^\alpha\Pne g\|_{L^1}
\le C_\alpha\|g\|_{L^1}
\qquad\text{if }|\alpha|\le3.
\end{equation}
\end{lem}

The kernel estimates and their consequences are proved in
Appendix~\ref{app:elementary}.

\begin{lem}[Products with fast periodic functions]\label{lem:products}
Let $\lambda\in\NN$, $f\in C^1(\TT^3)$, and
$g\in C^\infty(\TT^3)$. For $1\le q<\infty$,
\begin{equation}\label{eq:product-norm}
\|fg(\lambda\cdot)\|_{L^q}
\le\|g\|_{L^q}
\left(\|f\|_{L^q}+\sqrt3\lambda^{-1}\|\nabla f\|_\infty\right).
\end{equation}
If $\fint g=0$, then
\begin{equation}\label{eq:periodic-averaging}
\left|\int fg(\lambda\cdot)\right|
\le\sqrt3\lambda^{-1}\|\nabla f\|_\infty\|g\|_{L^1}.
\end{equation}

Let $K_0(z)=1+\dist(z,0)^{-2}$. If $\fint g=0$ and $\mathcal L$
is one of the operators of Lemma~\ref{lem:operators}, then pointwise
\begin{equation}\label{eq:periodic-commutator}
\begin{aligned}
&\bigl|\mathcal L[fg(\lambda\cdot)]
-f\mathcal L[g(\lambda\cdot)]\bigr|(x)\\
&\qquad\le C\lambda^{-1}\|f\|_{C^1}
\bigl(K_0*|(\Ro g)(\lambda\cdot)|\bigr)(x).
\end{aligned}
\end{equation}
Consequently,
\begin{equation}\label{eq:periodic-antidivergence-L1}
\|\mathcal L[fg(\lambda\cdot)]\|_{L^1}
\le C\lambda^{-1}\|f\|_{C^1}\|g\|_{L^1}.
\end{equation}
All bounds are uniform under translations of $g$ and hold
componentwise for vector or matrix arguments in the domain of
$\mathcal L$.
\end{lem}

\begin{proof}
The first two estimates follow by averaging over the periodic cells.
Divide $\TT^3$ into $\lambda^3$ cubes $Q_i$ of side $1/\lambda$, and
choose $x_i\in Q_i$ minimizing $|f|$. On each cube,
\[
|f(x)|\le|f(x_i)|+\sqrt3\lambda^{-1}\|\nabla f\|_\infty,
\]
while periodicity gives
\[
\int_{Q_i}|g(\lambda x)|^q\dd x=\lambda^{-3}\|g\|_{L^q}^q.
\]
Minkowski's inequality in $\ell^q$, together with
$\lambda^{-3}|f(x_i)|^q\le\int_{Q_i}|f|^q$, proves
\eqref{eq:product-norm}. When $g$ has zero mean, its integral against
the constant $f(x_i)$ vanishes on each cube. Subtracting this constant
and using the same oscillation bound proves
\eqref{eq:periodic-averaging}.

For the commutator, let $K$ be the kernel of $\mathcal L$. By
\eqref{eq:inverse-operators} and $\fint g=0$, $\D\Ro g=g$, so
\[
g(\lambda y)=\lambda^{-1}\D_y[(\Ro g)(\lambda y)].
\]
An integration by parts therefore yields
\[
\begin{aligned}
&\mathcal L[fg(\lambda\cdot)](x)
-f(x)\mathcal L[g(\lambda\cdot)](x)\\
&\quad=-\lambda^{-1}\int
\nabla_y\bigl[K(x-y)(f(y)-f(x))\bigr]
\cdot(\Ro g)(\lambda y)\dd y.
\end{aligned}
\]
To justify this identity, first integrate outside $B(x,h)$.
The boundary term is bounded by
\[
Ch^2h^{-2}h\|f\|_{C^1}\|\Ro g\|_\infty,
\]
which tends to zero. By \eqref{eq:operator-kernel},
\[
\bigl|\nabla_y[K(x-y)(f(y)-f(x))]\bigr|
\le C\|f\|_{C^1}K_0(x-y).
\]
This proves \eqref{eq:periodic-commutator}. Its right side has $L^1$
norm at most $C\lambda^{-1}\|f\|_{C^1}\|g\|_{L^1}$, because
$K_0$ is integrable and $\Ro$ is bounded on $L^1$ by
Lemma~\ref{lem:operators}. The remaining
term has the same bound by \eqref{eq:operator-scaling}, proving
\eqref{eq:periodic-antidivergence-L1}.
\end{proof}

\subsection{The far field}

Localization errors are evaluated away from the support of their
source. There the kernel is smooth, and repeated integration by parts
gives arbitrarily large gains in the frequency. The resulting bound
uses an integral norm of the profile and therefore does not lose a
power of its concentration.

\begin{lem}[Far field of oscillatory antidivergences]\label{lem:far}
Let $\mathcal L$ be one of the operators of
Lemma~\ref{lem:operators}, let $g\in C^\infty(\TT^3)$ have zero
mean, and let $f$ be a smooth scalar or vector field in the domain
of $\mathcal L$. Let $\lambda\in\NN$ and $0<d<1/4$. If
$\dist(x,\Supp f)\ge d$, then for every integer $m\ge0$,
\[
|\nabla^m\mathcal L[fg(\lambda\cdot)](x)|
\le C_m d^{-(m+6)}\lambda^{-4}
\|f\|_{C^4}\|\Delta^{-2}g\|_{L^1},
\]
uniformly under translations of $g$. For a unit vector $\bfe$ and
$0\le j\le3$,
$\|\Delta^{-2}(\bfe\cdot\nabla)^jg\|_{L^1}\le C\|g\|_{L^1}$
by \eqref{eq:operator-integrable-derivatives} from
Lemma~\ref{lem:operators}.

If $f=f(t,x)$ is smooth, then for $0\le j\le3$,
\begin{equation}\label{eq:time-derivatives}
\partial_t^j\bigl[f\,g(\lambda x-\lambda\omega t\bfe)\bigr]
=\sum_{i=0}^j\binom ji(\partial_t^{j-i}f)(-\lambda\omega)^i
\bigl((\bfe\cdot\nabla)^ig\bigr)(\lambda x-\lambda\omega t\bfe).
\end{equation}
Each differentiated profile satisfies the preceding integral bound,
uniformly under translations.
\end{lem}

\begin{proof}
Since $g$ has zero mean,
\[
g(\lambda y)=\lambda^{-4}\Delta_y^2[(\Delta^{-2}g)(\lambda y)].
\]
Let $K$ be the kernel of $\mathcal L$. Since $f$ vanishes on
$B(x,d)$, the function $y\mapsto\nabla^mK(x-y)f(y)$ is smooth.
Four integrations by parts give
\[
\begin{aligned}
\nabla^m\mathcal L[fg(\lambda\cdot)](x)
=\lambda^{-4}\int\Delta_y^2\bigl[\nabla^mK(x-y)f(y)\bigr]
(\Delta^{-2}g)(\lambda y)\dd y.
\end{aligned}
\]
By \eqref{eq:operator-kernel}, the bracket is bounded by
$Cd^{-(m+6)}\|f\|_{C^4}$. Integer periodic scaling preserves the
$L^1$ norm of $\Delta^{-2}g$, which proves the estimate.
The time-derivative formula is the chain rule. Each time derivative
costs a factor $\lambda\omega$ or a derivative of $f$, with no loss
in concentration.
\end{proof}

\subsection{Compactly supported antidivergences}\label{sec:compact}

Throughout this subsection $U=B(x_0,r_U)$ with
$0<r_U\le r_B<1/4$; we identify $U$ with a Euclidean ball.
Constants depend only on $r_B$ and a lower bound for $r_U$. In
Lemma~\ref{lem:ext} the radius is $r'$; in all later applications
$r_U\ge r_0\ge r_B/2$, where $r_0$ is the radius fixed after
Lemma~\ref{lem:ext}, so the constants depend only on $r_B$.

The range of a compactly supported divergence is constrained by
integration by parts. For a symmetric matrix field, symmetry also
forces the angular momentum of its divergence to vanish. These are exactly the
compatibility conditions used in the construction below.

\begin{lem}[Moment identities]\label{lem:moment-identities}
Let $F,a\in C^1_c(U;\RR^3)$, $q\in C^1_c(U)$ and
$M\in C^1_c(U;\mathrm{Sym}_3)$. Then
\begin{gather*}
\int F=-\int(x-x_0)\D F,\qquad
\int\nabla q=\int\D M=\int\curl a=0,\\
\int(x-x_0)\times\nabla q
=\int(x-x_0)\times\D M=0,\\
\int(x-x_0)\times\curl a=2\int a.
\end{gather*}
Moreover $\Delta v=\D(\nabla v+\nabla v^T)$ for every divergence-free
$v\in C^2$.
\end{lem}

\begin{proof}
All identities follow by integration by parts. In particular, symmetry
gives
\[
\begin{aligned}
\int\bigl[(x-x_0)\times\D M\bigr]_l
&=\int\varepsilon_{lab}(x-x_0)_a\partial_jM_{bj}\\
&=-\varepsilon_{lab}\int M_{ba}=0.
\end{aligned}
\]
The gradient case follows from $\nabla q=\D(q\Id)$. For the curl,
\[
\int\bigl[(x-x_0)\times\curl a\bigr]_l
=-\varepsilon_{lab}\varepsilon_{bad}\int a_d
=2\int a_l,
\]
using $\varepsilon_{lab}\varepsilon_{bad}=-2\delta_{ld}$.
\end{proof}

\begin{lem}[Bogovskii operator]\label{lem:bog}
There is a linear operator
\[
\Bog_U:\left\{h\in C^\infty_c(U):\int h=0\right\}
\longrightarrow C^\infty_c(U;\RR^3)
\]
such that
\[
\D\Bog_Uh=h,\qquad \int(x-x_0)\times\Bog_Uh=0,
\]
\[
\|\Bog_Uh\|_{C^k}\le C_k\|h\|_{C^k}\quad(k\ge0).
\]
It preserves smooth dependence on time and commutes with time
derivatives.
\end{lem}

\begin{proof}
Let $b_U(x)=r_U^{-3} b_1((x-x_0)/r_U)$, where
$b_1\in C^\infty_c(B(0,1/2))$ is fixed and radial, with $\int b_1=1$. The ball
$U$ is star-shaped with respect to every point of
$B(x_0,r_U/2)\supset\Supp b_U$. Bogovskii's operator
\cite{Bog79} (see also \cite[Section~III.3]{Galdi}) is
\[
\Bog_Uh(x)=\int_U h(y)\,\frac{x-y}{|x-y|^3}
\int_{|x-y|}^\infty b_U\Bigl(y+r\frac{x-y}{|x-y|}\Bigr)r^2\dd r\dd y.
\]
For $h\in C^\infty_c(U)$ with $\int h=0$ it satisfies
$\Bog_Uh\in C^\infty_c(U)$ and $\D\Bog_Uh=h$ \cite{Bog79},
\cite[Section~III.3]{Galdi}. The support property is also seen directly:
$\Bog_Uh(x)\ne0$ only if $x$ lies on a segment from a point of
$\Supp h$ to a point of $\Supp b_U$, and the union of these
segments is a compact subset of $U$.

To compute the angular momentum, write the inner integral using
$r\ge1$ and substitute $z=y+r(x-y)$. Since
$\int_1^\infty r^{-2}\dd r=1$, this gives
\[
\begin{aligned}
\int(x-x_0)\times\Bog_Uh
&=\left(\int(y-x_0)h(y)\dd y\right)
\times\left(\int(z-x_0)b_U(z)\dd z\right)=0.
\end{aligned}
\]
The last equality follows from the radial symmetry of $b_U$.

For the $C^k$ bound substitute $y=x-z$:
\[
\Bog_Uh(x)=\int h(x-z)\,K_U(x,z)\dd z,\qquad
K_U(x,z)=\frac{z}{|z|^3}\int_{|z|}^{2r_U} b_U\Bigl(x-z+r\frac z{|z|}\Bigr)r^2\dd r,
\]
where we used that the integrand vanishes for $r>2r_U$ when
$x-z\in U$. Derivatives in $x$ fall on $h$ or on $b_U$, and
$|\partial_x^\alpha K_U(x,z)|\le C_\alpha\|b_U\|_{C^{|\alpha|}} r_U^3|z|^{-2}$.
Since $|z|^{-2}$ is integrable on $B(0,2r_U)$, differentiation under
the integral gives
$|\partial^\alpha\Bog_Uh(x)|\le C_\alpha\|h\|_{C^{|\alpha|}}$, with
$C_\alpha$ depending only on $b_1$, $\alpha$ and bounds for
$r_U,r_U^{-1}$. The kernel does not depend on $t$, which gives
the last statement.
\end{proof}

\begin{lem}[Compactly supported symmetric antidivergence]\label{lem:sym}
There is a linear operator $\Sym_U$ such that, for every
$f\in C^\infty_c(U;\RR^3)$ satisfying
\[
\int f=0,\qquad \int(x-x_0)\times f=0,
\]
the field $\Sym_U f$ belongs to $C^\infty_c(U;\mathrm{Sym}_3)$ and obeys
\[
\D\Sym_U f=f,\qquad
\|\Sym_U f\|_{C^k}\le C_k\|f\|_{C^{k+1}}
\quad(k\ge0).
\]
The operator commutes with $\partial_t$.
\end{lem}

\begin{proof}
Since $\int f_i=0$, put $M_{ij}(x)=(\Bog_U f_i)_j(x)$, so that
$\partial_jM_{ij}= f_i$. Split $M=M^{\mathrm s}+A$ into symmetric and
antisymmetric parts, and let $n_k=\frac12\varepsilon_{kij}A_{ij}$. Then
$A_{ij}=\varepsilon_{ijk}n_k$ and $\D A=\curl n$. By
Lemma~\ref{lem:moment-identities},
\[
0=\int(x-x_0)\times f
=\int(x-x_0)\times\curl n=2\int n.
\]
Let $Y$ be the matrix with rows $Y_{l\cdot}=\Bog_Un_l$, so that
$\partial_jY_{lj}=n_l$, and define
\[
(\Sym_U f)_{ij}=M^{\mathrm s}_{ij}
+\varepsilon_{ikl}\partial_kY_{lj}+\varepsilon_{jkl}\partial_kY_{li}.
\]
This field is symmetric and satisfies
$\D\Sym_U f=f-\curl n+\curl n=f$. Lemma~\ref{lem:bog} gives
\[
\|\Sym_U f\|_{C^k}\le C\bigl(\|M\|_{C^k}+\|Y\|_{C^{k+1}}\bigr)
\le C\|f\|_{C^{k+1}}.
\]
\end{proof}

\subsection{Directions and lines}

To cancel the symmetric defect with nonnegative weights, we need
a positive decomposition into rank-one matrices. The following
explicit formula uses the two diagonals in each coordinate plane.

\begin{lem}[Geometric lemma]\label{lem:geometric}
For $1\le i<j\le3$, set
$\xi^\pm_{ij}=(\bfe_i\pm\bfe_j)/\sqrt2$. Every
$S\in\mathrm{Sym}_3$ with $|S-\Id|<1/5$ satisfies
\begin{equation}\label{eq:geometric-decomposition}
S=\sum_{1\le i<j\le3}\ \sum_\pm\Gamma^\pm_{ij}(S)\,
\xi^\pm_{ij}\otimes\xi^\pm_{ij},
\end{equation}
where the affine weights are
\[
\Gamma^\pm_{ij}(S)=S_{ii}+S_{jj}-\tfrac12\Tr S\pm S_{ij}
\ge\tfrac12-\tfrac52|S-\Id|>0.
\]
\end{lem}

\begin{proof}
The $ij$ entry on the right side of
\eqref{eq:geometric-decomposition} is
$\frac12(\Gamma^+_{ij}-\Gamma^-_{ij})=S_{ij}$. Its first diagonal
entry is
\[
\tfrac12(\Gamma^+_{12}+\Gamma^-_{12}
+\Gamma^+_{13}+\Gamma^-_{13})
=2S_{11}+S_{22}+S_{33}-\Tr S=S_{11},
\]
and the other diagonal entries follow in the same way. At $S=\Id$
each weight is $1/2$. Writing $\Tr S=S_{11}+S_{22}+S_{33}$, its
change is bounded by $\frac52|S-\Id|$, which proves positivity.
\end{proof}

This is an explicit instance of the geometric lemma of
\cite[Lemma~2.4]{DS17}. We use the finite sets of unit directions
\[
\begin{aligned}
\Lambda_R&=\{\pm\bfe_1,\pm\bfe_2,\pm\bfe_3\},\\
\Lambda_T&=\{\pm\xi^+_{ij},\pm\xi^-_{ij}:1\le i<j\le3\},
\qquad \Lambda=\Lambda_R\cup\Lambda_T.
\end{aligned}
\]
For the opposite directions, set
$\Gamma_{\pm\xi^+_{ij}}=\Gamma^+_{ij}$ and
$\Gamma_{\pm\xi^-_{ij}}=\Gamma^-_{ij}$. Every $k\in\Lambda$ is parallel
to a nonzero lattice vector, so $y+\RR k$ is a closed line in $\TT^3$.

\begin{lem}[Disjoint lines]\label{lem:lines}
There are points $y_k\in\TT^3$, $k\in\Lambda$, such that the closed
lines $L_k=y_k+\RR k$ are pairwise disjoint. Let $s_*>0$ be their
minimal distance.
\end{lem}

\begin{proof}
Choose the points successively. Given finitely many closed lines
$L_1,\dots,L_m$, a line $y+\RR k$ meets $L_i$ if and only if
$y\in L_i+\RR k$, which is the image of a smooth map from a
two-dimensional torus, hence a null set. Choose $y$ outside these null
sets.
\end{proof}

\subsection{The stationary profile}

We construct a stationary weight which is large along the lines
$L_k$ and has bounded $L^p$ norm. Its positive contribution will cover the density loss when the
new defects are corrected at the following stage.

For $\mu\ge1$ we define a smooth, positive, time-independent
function $\Psi_\mu$ on $\TT^3$. For each $k\in\Lambda$ fix linearly independent
$\mathbf n_{k,1},\mathbf n_{k,2}\in\ZZ^3$ orthogonal to $ k$, and put
\[
\begin{aligned}
F_k(x)&=\sum_{j=1}^2\sin^2\bigl(\pi\mathbf n_{k,j}\cdot(x-y_k)\bigr),\\
\Psi_\mu(x)&=1+\sum_{k\in\Lambda}\bigl(\mu^{-2}+F_k(x)\bigr)^{-1+\gamma/(2N)}.
\end{aligned}
\]
The zero set $Z_k$ of $F_k$ is a finite union of closed lines parallel
to $ k$, containing $L_k$. The linear map
$z\mapsto(\mathbf n_{k,1}\cdot z,\mathbf n_{k,2}\cdot z)$ is an
isomorphism from $ k^\perp$ onto $\RR^2$. Together with
$4\dist(s,\ZZ)^2\le\sin^2(\pi s)\le\pi^2\dist(s,\ZZ)^2$, this gives
\begin{equation}\label{eq:Fk}
c\dist(x,Z_k)^2\le F_k(x)\le C\dist(x,Z_k)^2,\qquad
|\nabla F_k|\le CF_k^{1/2},
\end{equation}
the last bound because $F_k$ is a sum of squares of smooth functions.
Let $s_0>0$ be smaller than half the minimal distance between two
distinct lines in $\bigcup_kZ_k$ that are parallel to the same
direction.

For the moving profiles, fix scalar functions
$\zeta\in C^\infty_c(B(0,1/16);[0,\infty))$ and
$\chi_*\in C^\infty_c(B(0,1/8);[0,1])$ on $\RR^3$, with
$\int_{\RR^3}\zeta=1$ and $\chi_*=1$ on $B(0,1/16)$. Define
\begin{equation}\label{eq:concentrated-bump}
\zeta_\mu(x)=\sum_{j\in\ZZ^3}\mu^3\zeta(\mu(x-j)).
\end{equation}
Thus $\zeta_\mu$ is smooth, nonnegative, and periodic, with unit
spatial mean.

\begin{lem}[Stationary profile]\label{lem:profile}
The profile $\Psi_\mu$ is at least one and satisfies, for every integer
$j\ge0$,
\begin{equation}\label{eq:profile-integrability}
\|\Psi_\mu\|_{L^p}+\|\Psi_\mu\|_{L^1}\le C,
\qquad
\|\nabla^j\Psi_\mu\|_\infty\le C_j\mu^{2+j}.
\end{equation}
Its ball averages are controlled by its value at the center:
\begin{equation}\label{eq:profile-average}
\fint_{B(x,h)}\Psi_\mu(\lambda y)\dd y
\le C\Psi_\mu(\lambda x).
\end{equation}
Consequently, for $K_0(z)=1+\dist(z,0)^{-2}$,
\begin{equation}\label{eq:profile-convolution}
K_0*\Psi_\mu(\lambda\cdot)\le C\Psi_\mu(\lambda\cdot).
\end{equation}
For every $k\in\Lambda$, the profile also satisfies
\begin{equation}\label{eq:profile-tube}
\Psi_\mu(x)\ge c\mu^{2-\gamma/N}
\quad\text{if}\quad \dist(x,L_k)\le\frac1{8\mu}.
\end{equation}

Let $f$ be a smooth scalar function on $\TT^3$, or a smooth vector
field when $\mathcal L=\Rs$. If $z\in L_k$ and
$g=\zeta_\mu(\cdot-z)-1$, then, for $\mathcal L\in\{\Ro,\Rs\}$,
\begin{equation}\label{eq:profile-antidivergence}
\bigl|\mathcal L[fg(\lambda\cdot)](x)\bigr|
\le C\lambda^{-1}\|f\|_{C^1}\mu^{\gamma/N}\Psi_\mu(\lambda x).
\end{equation}
If instead $g=\Psi_\mu-\fint\Psi_\mu$, then
\begin{equation}\label{eq:profile-self-antidivergence}
\bigl|\Ro[fg(\lambda\cdot)](x)\bigr|
\le C\lambda^{-1}\|f\|_{C^1}\Psi_\mu(\lambda x).
\end{equation}
All constants are uniform in $\mu\ge1$, integers $\lambda\ge1$,
$x\in\RR^3$ and $h>0$, using the periodic extension. They may depend
on the fixed exponents $\gamma,N,p$.
\end{lem}

\begin{proof}
The estimates follow from the two-dimensional transverse geometry of
the lines. Put
\[
\psi_k(x)=\bigl(\mu^{-2}+\dist(x,Z_k)^2\bigr)^{-1+\gamma/(2N)}.
\]
By \eqref{eq:Fk}, $\Psi_\mu$ is comparable to $1+\sum_k\psi_k$.
We first estimate the integral and local averages of each $\psi_k$,
then apply these estimates to the antidivergence kernels.

\proofstep{Step 1: Integral bounds and ball averages}
Since $p-1<\gamma/N^{10}$, we have
$(2-\gamma/N)q<2-\gamma/(2N)$ for $q=1,p$. Near each line of
$Z_k$, cylindrical coordinates therefore give
\[
\int\psi_k^q
\le C\int_0^{s_0}(\mu^{-2}+t^2)^{-(1-\gamma/(2N))q}t\dd t+C
\le \frac{CN}{\gamma}.
\]
Away from these lines the profile is uniformly bounded. For its
derivatives, the chain rule and $|\nabla F_k|\le CF_k^{1/2}$ give
\[
\left|\nabla^j(\mu^{-2}+F_k)^{-1+\gamma/(2N)}\right|
\le C_j(\mu^{-2}+F_k)^{-1+\gamma/(2N)-j/2}
\le C_j\mu^{2+j}.
\]
Indeed, first derivatives of $F_k$ contribute powers of $F_k^{1/2}$,
while its higher derivatives are bounded. This proves
\eqref{eq:profile-integrability}.

To prove \eqref{eq:profile-average}, a change of variables reduces the
claim to $\lambda=1$. Write $D=\dist(x,Z_k)$. If $h\le D/2$, the
distance to $Z_k$ is at least $D/2$ on $B(x,h)$, so
$\psi_k(y)\le C\psi_k(x)$ there. If $D/2<h\le s_0/6$, the ball is
contained in a ball of radius $3h$ centered on a single line of $Z_k$.
Therefore
\[
\begin{aligned}
\fint_{B(x,h)}\psi_k
&\le Ch^{-2}\int_0^{3h}(\mu^{-2}+t^2)^{-1+\gamma/(2N)}t\dd t\\
&\le C(\mu^{-2}+h^2)^{-1+\gamma/(2N)}
\le C\psi_k(x).
\end{aligned}
\]
The second inequality follows by considering $h\le\mu^{-1}$ and
$h\ge\mu^{-1}$ separately, using $2-\gamma/N<2$. For $h>s_0/6$, cover
$B(x,h)$ by at most $C(1+h)^3$ unit cubes. The integral bound controls
the average, while periodicity gives a fixed positive lower bound for
$\psi_k(x)$. Summing over $k$ proves the required ball-average estimate.

\proofstep{Step 2: Convolution and the lower bound on the tubes}
On the fundamental domain $x+[-1/2,1/2)^3$, the kernel $K_0$ is
bounded by $C|z|^{-2}$. On the annulus
$2^{-i-1}<|z|\le2^{-i}$, \eqref{eq:profile-average} bounds the
convolution by
\[
C2^{2i}2^{-3i}\Psi_\mu(\lambda x).
\]
Summing over $i\ge0$ gives \eqref{eq:profile-convolution}. If
$\dist(x,L_k)\le1/(8\mu)$, then \eqref{eq:Fk} gives
$F_k(x)\le C\mu^{-2}$. This proves \eqref{eq:profile-tube}.

\proofstep{Step 3: Pointwise antidivergence estimates}
Let $g=\zeta_\mu(\cdot-z)-1$ with $z\in L_k$. The kernels of
$\mathcal L$ and $\Ro$ have zero mean and are bounded by $CK_0$
by Lemma~\ref{lem:operators} and \eqref{eq:operator-kernel}.
The bump $\zeta_\mu(\cdot-z)$ has unit mass and is supported in
$B(z,1/(16\mu))$. Integrating the kernel over this bump gives a bound
$C\mu^2$ when $\dist(x,z)\le1/(8\mu)$, and
$C(1+\dist(x,z)^{-2})$ otherwise. Hence
\[
\begin{aligned}
|\mathcal Lg(x)|+|\Ro g(x)|
&\le C\bigl(\mu^{-2}+\dist(x,z)^2\bigr)^{-1}\\
&\le C\mu^{\gamma/N}\Psi_\mu(x),
\end{aligned}
\]
where the last inequality uses $z\in L_k\subset Z_k$.
Applying \eqref{eq:operator-scaling} gives
\[
|f\mathcal L[g(\lambda\cdot)](x)|
\le C\lambda^{-1}\|f\|_\infty\mu^{\gamma/N}\Psi_\mu(\lambda x).
\]
By \eqref{eq:periodic-commutator}, the commutator is bounded by
\[
C\lambda^{-1}\|f\|_{C^1}\mu^{\gamma/N}
K_0*\Psi_\mu(\lambda\cdot).
\]
Together with \eqref{eq:profile-convolution}, this proves
\eqref{eq:profile-antidivergence}.

For $g=\Psi_\mu-\fint\Psi_\mu$, the zero mean of the kernel gives
\[
|\Ro g|\le CK_0*\Psi_\mu\le C\Psi_\mu.
\]
Now \eqref{eq:operator-scaling}, \eqref{eq:periodic-commutator}, and
\eqref{eq:profile-convolution} give
\eqref{eq:profile-self-antidivergence} in the same way.
\end{proof}

\section{One step of the iteration: the construction}\label{sec:construction}

We construct $(\rho^+,u^+)=(\rho+\theta,u+w)$ from a smooth
approximate solution. The principal perturbations cancel the old
defects. The stationary correction $\theta_d$ adds density where
the new defects, including the viscous error from $\nabla w$, will
cause a subtraction at the following stage. We then arrange the
continuity and momentum equations successively.

\subsection{Setting and statement}

Fix $n\ge0$ and use the sequences $\lambda_n,\delta_n$ chosen above.
Set
\begin{equation}\label{eq:perturbation-scales}
\lambda=\lambda_{n+1},\qquad\mu=\lambda^N,\qquad\omega=\delta_n^{-4}.
\end{equation}
The integers $\lambda,\mu$ are the frequency and concentration, and
$\omega>0$ is the translation speed.
We are given a compact interval $I'$ and an open interval $J'$ with
$I'\Subset J'\subset J$, and put $\tau=\dist(I',\partial J')$. We are
given radii $r_0\le r_-< r_+\le r_B$, and put
$d=(r_+- r_-)/(64N+40)$. Finally, we are given a smooth solution
$Z=(\rho,u,\Pi,R,T)$ of the system of Definition~\ref{def:system} such that $(R,T)$ is supported
in $I'\times\overline B(x_0,r_-)$. At the initial stage $n=0$,
we take $Z=Z_*$.

\begin{prop}[One step]\label{prop:step}
Assume that
\begin{equation}\label{eq:input-bounds}
\begin{gathered}
\|(\rho,u,R,T)\|_{C^{20}_{t,x}}\le\lambda_n^{N^4},\qquad
\|u\|_\infty\le2\delta_n^{-1},\qquad \tau^{-1}+d^{-1}\le\lambda_n,\\
\|(R,T)\|_{C_tL^p}\le2\delta_n^8\quad\text{if }n\ge1.
\end{gathered}
\end{equation}
There is $a_0$, depending only on the fixed parameters, $r_0$, $r_B$
and $Z_*$, with the following property. If $a\ge a_0$ and $Z$ is as
above, there is a smooth solution
$Z^+=(\rho^+,u^+,\Pi^+,R^+,T^+)$ of \eqref{eq:defect} from Definition~\ref{def:system}. Both
$Z^+-Z$ and $(R^+,T^+)$ are supported in
\[
K'\times B\bigl(x_0,\tfrac12(r_-+ r_+)\bigr)
\]
for a compact interval $K'\subset J'$.

The new defects satisfy
\begin{equation}\label{eq:step-defects}
\|(R^+,T^+)\|_{C_tL^p}\le\delta_{n+1}^8,
\end{equation}
and the total density perturbation $\theta=\rho^+-\rho$ satisfies
\begin{equation}\label{eq:step-positivity}
\begin{aligned}
\theta\ge{}&-\delta_n^2|R|-\delta_n^4|T|-\delta_n^2\\
&+\delta_{n+1}^2|R^+|+\delta_{n+1}^4|T^+|.
\end{aligned}
\end{equation}
The density and velocity increments satisfy
\begin{equation}\label{eq:step-increments}
\begin{aligned}
\|\rho^+-\rho\|_{C_tL^p}&\le\delta_n^2,\\
\|u^+-u\|_{C_tH^s}&\le\delta_{n+1}^{16},\qquad
\|u^+\|_\infty\le\delta_{n+1}^{-1}.
\end{aligned}
\end{equation}
If $n\ge1$, the nonlinear quantities and pressure also satisfy
\begin{equation}\label{eq:step-products}
\begin{aligned}
&\|\rho^+u^+-\rho u\|_{C_tL^1}
+\|\rho^+u^+\otimes u^+-\rho u\otimes u\|_{C_tL^1}\\
&\qquad+\|\Pi^+-\Pi\|_{L^1_{t,x}}\le\delta_n^2.
\end{aligned}
\end{equation}
Finally,
\begin{equation}\label{eq:step-derivatives}
\|(\rho^+,u^+,R^+,T^+)\|_{C^3_{t,x}}
+\|\Pi^+-\Pi\|_{C^3_{t,x}}\le\tfrac12\lambda^{N^2}.
\end{equation}
\end{prop}

The positive terms in \eqref{eq:step-positivity} allow for the
density removed when correcting $R^+,T^+$ at the following stage.
The rest of this section defines $Z^+$ and verifies the equations;
Section~\ref{sec:estimates} proves the estimates. In both sections, $C$
denotes constants depending only on the fixed parameters, $r_0$, $r_B$
and, where indicated by $C_*$, on $Z_*$; they never depend on $n$ or
$a$. The lower bound for $a$ is uniform in $n$.

\subsection{Cutoffs for the principal perturbations}\label{sec:balls}

We first localize the coefficients and the velocity perturbation.
For $j\ge0$ write
\[
B_j=B(x_0,r_-+2jd).
\]
All spatial cutoffs below take values in $[0,1]$ and satisfy
$\|\cdot\|_{C^l}\le C_ld^{-l}$ for every integer $l\ge0$.
Choose
\begin{equation}\label{eq:principal-cutoffs}
\begin{aligned}
&\chi_R\in C^\infty_c(B_1),
&&\chi_R=1\quad\text{on }B(x_0,r_-+d),\\
&\chi_u\in C^\infty_c(B_3),
&&\chi_u=1\quad\text{on }B_2.
\end{aligned}
\end{equation}
Let $\eta,\eta_u$ be the time cutoffs of Lemma~\ref{lem:cutoffs}
for $I'\Subset J'$. Thus $\chi_u=1$ near $\Supp\chi_R$ and
$\eta_u=1$ near $\Supp\eta$. These choices ensure that the
velocity cutoffs do not alter the cancellation on the density
profiles.

\subsection{Cancellation of the defects}

We seek nonnegative scalar weights $a_k$ on $[0,1]\times\TT^3$
whose first and second velocity moments cancel $R$ and $T$. Set the constant vectors
$V_k=\omega k$ for the directions $k\in\Lambda$ fixed after
Lemma~\ref{lem:geometric}. Proposition~\ref{prop:coefficients}
below states the cancellation identities.
The coordinate directions will produce the first moment. Opposite
pairs from $\Lambda_T$ have zero first moment and allow us to adjust
the symmetric second moment independently.

To keep the weights nonnegative and smooth at zeros of the defects,
define the nonnegative scalar function
\begin{equation}\label{eq:coefficients}
\alpha(t,x)=\bigl(4\delta_n^{16}\chi_R(x)^4+|R(t,x)|^2+\omega^{-2}|T(t,x)|^2\bigr)^{1/2}.
\end{equation}
Define the symmetric matrix field of trace three
\begin{equation}\label{eq:normalized-stress}
S(t,x)=\Id+\frac4{\omega^2}
\Bigl(u\otimes u-\frac{|u|^2}{3}\Id\Bigr)
-\frac{T-\frac13\Tr(T)\Id}{6\omega\alpha}.
\end{equation}
The quotients $R/\alpha$ and $T/\alpha$ are set to zero where
$R=T=0$. Using the affine functions of
Lemma~\ref{lem:geometric}, define
\begin{equation}\label{eq:density-coefficients}
\begin{aligned}
a_k(t,x)&=\frac{\eta(t)^2\alpha}{\omega}
\left[1+\left(\frac{12u}{\omega}-\frac{R}{2\alpha}\right)\cdot k\right],
&&k\in\Lambda_R,\\
a_k(t,x)&=\frac{3\eta(t)^2\alpha}{\omega}\Gamma_k(S),
&&k\in\Lambda_T.
\end{aligned}
\end{equation}
The scalar part of the second moment is absorbed into the pressure
through the scalar function
\begin{equation}\label{eq:pressure-coefficient}
\pi(t,x)=\eta(t)^2\left[8\omega\alpha
\left(1-\frac{|u|^2}{\omega^2}\right)+\frac13\Tr T\right].
\end{equation}
Here and below all direction indices range over $\Lambda$.

Define the smooth scalar fields
\begin{equation}\label{eq:transport-coefficient}
f_k(t,x)=(\partial_t+V_k\cdot\nabla)a_k(t,x).
\end{equation}
They are supported in $\Supp\eta\times B_1$.

\begin{prop}[Coefficients]\label{prop:coefficients}
For sufficiently large $a>0$, the functions $a_k$ and $\pi$ are smooth and supported
in $\Supp\eta\times B_1$, and $a_k\ge0$. They satisfy
\begin{equation}\label{eq:coefficient-moments}
\begin{aligned}
\sum_k a_k(V_k-u)&=-R,\\
\sum_k a_k(V_k\otimes V_k-u\otimes u)&=\pi\Id-T.
\end{aligned}
\end{equation}
and
\begin{equation}\label{eq:coefficient-sum}
\sum_k a_k=\frac{24\eta^2\alpha}{\omega}.
\end{equation}
Their pointwise sizes obey
\begin{equation}\label{eq:coefficient-bounds}
\begin{aligned}
\sum_k a_k
&\le\frac{24}{\omega}
\left(2\delta_n^8+|R|+\frac{|T|}{\omega}\right),\\
|\pi|&\le C\bigl(\omega\delta_n^8+\omega|R|+|T|\bigr).
\end{aligned}
\end{equation}
For $0\le j\le10$, they also satisfy
\begin{equation}\label{eq:coefficient-derivatives}
\|a_k\|_{C^j_{t,x}}\le\delta_{n+1}^{-1}\omega^{-1},
\qquad
\|\pi\|_{C^j_{t,x}}\le\delta_{n+1}^{-1}\omega.
\end{equation}
The material derivatives $f_k$ satisfy
\begin{equation}\label{eq:transport-coefficient-bounds}
\|f_k\|_{C^j_{t,x}}
\le C\omega\|a_k\|_{C^{j+1}_{t,x}}
\le C\delta_{n+1}^{-1},\qquad 0\le j\le9.
\end{equation}
\end{prop}

\begin{proof}
We verify the algebraic identities and positivity first, then estimate
the derivatives of the coefficients.

\proofstep{Step 1: Positivity and smoothness of the weights}
The open sets
\[
O_1=B(x_0,r_-+d),\qquad
O_2=\TT^3\setminus\overline B(x_0,r_-)
\]
cover the torus. By \eqref{eq:principal-cutoffs} and
\eqref{eq:coefficients}, on $O_1$ the cutoff $\chi_R$ is one and
$\alpha\ge2\delta_n^8$. On $O_2$, the defects vanish and
$\alpha=2\delta_n^8\chi_R^2$. It follows that $\alpha$, $R/\alpha$
and $T/(\omega\alpha)$ are smooth. In addition,
\[
|R|\le\alpha,\qquad |T|\le\omega\alpha.
\]
By \eqref{eq:input-bounds} and $\omega=\delta_n^{-4}$, for sufficiently large $a>0$,
\[
\frac{12|u|}{\omega}\le24\delta_n^3\le\frac14,
\qquad
\frac{4|u|^2}{\omega^2}\le16\delta_n^6\le10^{-3}.
\]
Consequently,
\[
\left|\frac{12u}{\omega}-\frac{R}{2\alpha}\right|\le\frac34,
\qquad
|S-\Id|\le10^{-3}+\frac16<0.17.
\]
The coordinate weights are therefore nonnegative, as are the geometric
weights by Lemma~\ref{lem:geometric}. Their spatial support is contained
in $B_1$, since both the defects and $\chi_R$ vanish outside this
ball. The factor $\eta^2$ gives the temporal support.

\proofstep{Step 2: The continuity defect}
For $i=1,2,3$, \eqref{eq:density-coefficients} gives
\[
a_{\bfe_i}+a_{-\bfe_i}=\frac{2\eta^2\alpha}{\omega}
\]
and
\[
\omega(a_{\bfe_i}-a_{-\bfe_i})
=\eta^2\left(\frac{24\alpha}{\omega}u-R\right)\cdot\bfe_i.
\]
For $k\in\Lambda_T$, the weights are equal on opposite directions, so
their first moments cancel. Taking the trace in
Lemma~\ref{lem:geometric} and using $\Tr S=3$ gives
\[
\sum_{k\in\Lambda_T}a_k
=\frac{6\eta^2\alpha}{\omega}\Tr S
=\frac{18\eta^2\alpha}{\omega}.
\]
Adding the two groups yields
\[
\sum_k a_k=\frac{24\eta^2\alpha}{\omega},
\qquad
\sum_k a_kV_k=\left(\sum_k a_k\right)u-\eta^2R.
\]
As $\eta=1$ wherever $R\ne0$, this is the first identity in
\eqref{eq:coefficient-moments}.

\proofstep{Step 3: The stress}
By \eqref{eq:density-coefficients} and
Lemma~\ref{lem:geometric}, the two groups contribute
\[
\sum_{k\in\Lambda_R}a_kV_k\otimes V_k
=2\eta^2\omega\alpha\Id,
\qquad
\sum_{k\in\Lambda_T}a_kV_k\otimes V_k
=6\eta^2\omega\alpha S.
\]
Insert \eqref{eq:normalized-stress} and use $\eta^2T=T$ to obtain
\[
\begin{aligned}
\sum_k a_kV_k\otimes V_k
={}&8\eta^2\omega\alpha\Id
+\frac{24\eta^2\alpha}{\omega}
\left(u\otimes u-\frac{|u|^2}{3}\Id\right)\\
&-\left(T-\frac{\Tr T}{3}\Id\right)\\
={}&\left(\sum_k a_k\right)u\otimes u+\pi\Id-T.
\end{aligned}
\]
The last equality uses \eqref{eq:coefficient-sum} and
\eqref{eq:pressure-coefficient}, and proves the second identity in
\eqref{eq:coefficient-moments}. Finally, \eqref{eq:coefficients} gives
\[
\alpha\le2\delta_n^8\chi_R^2+|R|+\frac{|T|}{\omega},
\]
which gives \eqref{eq:coefficient-bounds}, using $|u|\le\omega$ for
the pressure bound.
\proofstep{Step 4: Derivative estimates}
All norms in this step are $C^{10}_{t,x}$ norms. By
\eqref{eq:input-bounds}, the input fields, $\delta_n^{-8}$, and the
inverse support margins are bounded by $\lambda_n^{N^4}$.
Lemma~\ref{lem:cutoffs} and the derivative bounds for the spatial
cutoffs give
\[
\|\eta^2\|+\|\chi_R^i\|\le C\lambda_n^{10N^4}
\qquad(i=2,4).
\]
On $O_1$, the norm of $(R,T/\omega)/(2\delta_n^8)$ is at most
$C\lambda_n^{2N^4}$. The maps
$z\mapsto z/(1+|z|^2)^{1/2}$ and
$z\mapsto(1+|z|^2)^{1/2}$ have bounded positive-order derivatives.
Lemma~\ref{lem:calculus} therefore gives the coarse bound
\[
\|\alpha\|+\|R/\alpha\|+\|T/(\omega\alpha)\|
\le C\lambda_n^{30N^4}.
\]
On $O_2$, the same bound follows from
$\alpha=2\delta_n^8\chi_R^2$ and $R=T=0$. The affine weights of
Lemma~\ref{lem:geometric} and \eqref{eq:normalized-stress} then give
$\|S\|+\|\Gamma_k(S)\|\le C\lambda_n^{30N^4}$.
Applying the product estimate of Lemma~\ref{lem:calculus} to
\eqref{eq:density-coefficients} and \eqref{eq:pressure-coefficient},
we obtain
\[
\omega\|a_k\|+\omega^{-1}\|\pi\|
\le C\lambda_n^{100N^4}
\le\lambda^{\gamma/2}\le\delta_{n+1}^{-1}.
\]
Here $\lambda=\lambda_n^b$ and $100N^4<\gamma b/2$, so the
penultimate inequality holds
for sufficiently large $a>0$. This proves \eqref{eq:coefficient-derivatives}.
Since $V_k$ is constant, the product rule gives
\eqref{eq:transport-coefficient-bounds}.
\end{proof}

\subsection{Moving profiles and the velocity perturbation}

The velocity perturbation must equal $V_k-u$ wherever the $k$ th
density profile is present. We obtain this identity by taking a local
vector potential of $V_k-u$ and applying a cutoff inside a curl. The
curl preserves incompressibility, and the cutoff is constant on the
smaller region occupied by the profile.

Let $y_k$, $L_k$ and $s_*$ be as in Lemma~\ref{lem:lines}. We use
the bump $\zeta$, cutoff $\chi_*$, and periodic profile
$\zeta_\mu$ fixed before Lemma~\ref{lem:profile}. Define the
smooth scalar profiles
\begin{equation}\label{eq:traveling-profiles}
\Theta_k(t,x)=\zeta_\mu\bigl(\lambda x-\lambda\omega t k-y_k\bigr),\qquad
V_k=\omega k .
\end{equation}
Choose representatives $y_k\in\RR^3$. For $j\in\ZZ^3$ let $X_{k,j}(t)=\lambda^{-1}(j+y_k)+\omega t k$.
Note that $\lambda X_{k,j}(t)\in L_k$ for all $t$, since $L_k$ is
invariant under translations along $ k$. Using the periodic lift of
$u$, define
\begin{equation}\label{eq:velocity-blocks}
W_k(t,x)=\eta_u(t)\sum_{j\in\ZZ^3}\curl_x\Bigl[\chi_u(x)\,
\chi_*\bigl(\lambda\mu(x-X_{k,j})\bigr)A_{k,j}(t,x)\Bigr],
\end{equation}
\[
A_{k,j}(t,x)=\int_0^1s\,v_k\bigl(t,X_{k,j}+s(x-X_{k,j})\bigr)\times(x-X_{k,j})\dd s,
\qquad v_k(t,x)=\omega k-u(t,x).
\]
The vector field $v_k$ is smooth and divergence-free, and $A_{k,j}$
is its local vector potential, as verified below. The sum in
\eqref{eq:velocity-blocks} is locally finite and defines a smooth
periodic vector field. Put $w(t,x)=\sum_kW_k(t,x)$ and $u^+=u+w$.

\begin{lem}[Moving profiles and velocity perturbation]\label{lem:blocks}
The profiles $\Theta_k$ satisfy
\begin{equation}\label{eq:moving-profile-bounds}
\Theta_k\ge0,\qquad \fint\Theta_k=1,\qquad
\|\Theta_k\|_{L^q}\le C\mu^{3-3/q}\quad(1\le q\le\infty).
\end{equation}
\begin{equation}\label{eq:profile-transport}
\partial_t\Theta_k+V_k\cdot\nabla\Theta_k=0.
\end{equation}
For sufficiently large $a>0$, each $W_k$ is divergence-free and has
zero spatial mean.
Its support is contained in
\[
B_3\cap\bigcup_j B\bigl(X_{k,j}(t),(8\lambda\mu)^{-1}\bigr)
\subset\mathcal C_k,
\]
where the time-independent tubes are
\[
\mathcal C_k=\{x\in\TT^3:\dist(\lambda x,L_k)\le1/(8\mu)\}.
\]
The tubes $\mathcal C_k$ are pairwise disjoint, and
\begin{equation}\label{eq:velocity-support}
|\Supp W_k(t)|\le C\mu^{-3}.
\end{equation}
On the density profiles the total velocity is exact:
\begin{equation}\label{eq:velocity-on-support}
a_k\Theta_k(u^+-V_k)=0.
\end{equation}
For $0\le j\le10$,
\begin{equation}\label{eq:velocity-derivatives}
\begin{aligned}
|\nabla^jW_k|
&\le C\omega(\lambda\mu)^j\mathbf1_{\Supp W_k},\\
\|W_k\|_{C^j_{t,x}}
&\le C\omega(\lambda\mu\omega)^j.
\end{aligned}
\end{equation}
Finally, the full perturbation satisfies
\begin{equation}\label{eq:velocity-norms}
|w|\le C\omega,\qquad
\|w\|_{C_tH^s}\le\delta_{n+1}^{16}.
\end{equation}
\end{lem}

\begin{proof}
The main point is the identity \eqref{eq:velocity-on-support}.
The curl construction gives zero divergence automatically; we show
that it agrees with $V_k-u$ wherever $a_k\Theta_k$ is nonzero. The
Sobolev estimate then follows from the size and support of the blocks.

\proofstep{Step 1: Support and incompressibility}
The properties of $\Theta_k$ follow from its definition and integer
periodic scaling. The support statement follows from \eqref{eq:velocity-blocks}, because
$\Supp\chi_*\subset B_{1/8}$ and $\lambda X_{k,j}\in L_k$. The tubes are
disjoint once $1/(4\mu)<s_*$. The balls around $X_{k,j}(t)$,
$j\in\{0,\dots,\lambda-1\}^3$, are $\lambda^3$ balls of radius
$(8\lambda\mu)^{-1}$ in $\TT^3$, which proves
\eqref{eq:velocity-support}. Each $W_k$ is the curl of a smooth
periodic field, which gives zero divergence and zero mean.

\proofstep{Step 2: Exact transport on each profile}
For a divergence-free field $v$ and a point $X$, the identity
$\curl(a\times b)=a\D b-b\D a+(b\cdot\nabla)a-(a\cdot\nabla)b$ gives
\[
\begin{aligned}
&\curl_x\bigl[s\,v(X+s(x-X))\times(x-X)\bigr]\\
&\quad=2s\,v(X+s(x-X))
+s^2\bigl((x-X)\cdot\nabla\bigr)v(X+s(x-X))\\
&\quad=\frac{\dd}{\dd s}\bigl[s^2v(X+s(x-X))\bigr].
\end{aligned}
\]
Integrating in $s$ gives $\curl_xA_{k,j}=v_k$. On $\Supp\Theta_k$ the point $x$ lies in some
$B(X_{k,j},(16\lambda\mu)^{-1})$, where $\chi_*(\lambda\mu(\cdot-X_{k,j}))=1$
on a neighbourhood. By Proposition~\ref{prop:coefficients},
$\Supp a_k\subset\Supp\eta\times B_1$. Thus
\eqref{eq:principal-cutoffs} and Lemma~\ref{lem:cutoffs} give
$\chi_u=\eta_u=1$ on a neighbourhood of $\Supp a_k$. Thus $W_k=v_k=\omega k-u$ on $\Supp(a_k\Theta_k)$, while
$W_{k'}=0$ there for $k'\ne k$ by disjointness of the tubes. Hence $u^+=V_k$ there.

\proofstep{Step 3: Derivative estimates}
Fix $X=X_{k,j}(t)$ and $|x-X|\le(8\lambda\mu)^{-1}$. By
\eqref{eq:input-bounds},
$\|u\|_{C^h_{t,x}}\le\lambda_n^{N^4}\le\omega(\lambda\mu)^h$
for $1\le h\le20$, and $|v_k|\le2\omega$; hence
$\|\nabla^hv_k\|_\infty\le2\omega(\lambda\mu)^h$. Differentiating
$A_{k,j}$ under the integral gives
\[
\begin{aligned}
|\nabla^hA_{k,j}|
&\le C\bigl(|x-X|\|\nabla^hv_k\|_\infty
+\|\nabla^{h-1}v_k\|_\infty\bigr)\\
&\le C\omega(\lambda\mu)^{h-1}.
\end{aligned}
\]
For $h=0$ the second term is absent. Writing
$\chi_{*,j}=\chi_*(\lambda\mu(x-X))$, the product rule gives
\[
\begin{aligned}
\curl_x[\chi_u\chi_{*,j}A_{k,j}]
={}&\chi_u\chi_{*,j}v_k
+\chi_{*,j}\nabla\chi_u\times A_{k,j}\\
&+\chi_u\lambda\mu(\nabla\chi_*)(\lambda\mu(x-X))\times A_{k,j},
\end{aligned}
\]
and $\|\chi_u\|_{C^i}\le Cd^{-i}\le C(\lambda\mu)^i$. Leibniz' rule
gives the spatial bound in \eqref{eq:velocity-derivatives}, since the balls for different $j$ are
disjoint. For time derivatives note $X'_{k,j}=\omega k$: each time
derivative falling on $\chi_{*,j}$ or on $x-X$ costs at most $C\lambda\mu\omega$,
on $v_k$ (through $u$ and $X$) at most $C\omega\lambda\mu$ relative to the
bounds above, and on $\eta_u$ at most $C\tau^{-1}\le C\lambda_n$. This gives
the second bound.

\proofstep{Step 4: The Sobolev estimate}
The spatial bound in \eqref{eq:velocity-derivatives} gives
$|W_k|\le C\omega$. Combining it with
\eqref{eq:velocity-support}, we obtain
\[
\|W_k\|_{L^2}\le C\omega\mu^{-3/2},\qquad
\|\nabla^2W_k\|_{L^2}\le C\omega(\lambda\mu)^2\mu^{-3/2}.
\] Interpolation
gives
\[
\|w\|_{H^s}\le C\omega\mu^{-3/2}(\lambda\mu)^s
=C\lambda^{4\gamma/b+s+N(s-3/2)}\le C\lambda^{-1}
\le\delta_{n+1}^{16}
\]
for sufficiently large $a>0$, since $N(3/2-s)>s+2$ and
$16\gamma<1$.
\end{proof}

\subsection{The density perturbation}

The principal density increment is
\begin{equation}\label{eq:principal-density}
\theta_p(t,x)=\sum_k a_k(t,x)(\Theta_k(t,x)-1).
\end{equation}
Since $a_k,\Theta_k\ge0$, its negative part is bounded by
$\sum_k a_k$, which is small by Proposition~\ref{prop:coefficients}.
We add density so that the next perturbation preserves positivity,
then restore the mass and first spatial moment.

\subsubsection*{The stationary density correction}

Applying Proposition~\ref{prop:coefficients} at the following
stage shows that the density removed by its principal perturbation
is bounded by
\[
C\bigl(\delta_{n+1}^4|R^+|+\delta_{n+1}^8|T^+|
+\delta_{n+1}^{12}\bigr).
\]
For sufficiently large $a>0$, the defect terms are bounded by
$\delta_{n+1}^2|R^+|+\delta_{n+1}^4|T^+|$, as required by the
density estimate in Proposition~\ref{prop:step}. The new stress contains
$-(\nabla w+\nabla w^T)$, whose supremum can grow like
$\lambda\mu\omega$. Its large cutoff gradients can occur outside
the density bumps, so those bumps do not control the density loss
at the next stage.

The stationary profile of Lemma~\ref{lem:profile} covers the whole
periodic line traversed by each moving bump. Its lower bound on the
tubes and Lemma~\ref{lem:blocks} give
\[
|\nabla w|\le C\lambda\mu\omega\,
\mu^{-2+\gamma/N}\Psi_\mu(\lambda x)
=C\lambda^{1-N+\gamma}\omega\Psi_\mu(\lambda x).
\]
Thus the profile can supply enough density on the cutoff regions
at small $L^p$ cost. It also controls the oscillatory
antidivergences by Lemma~\ref{lem:profile}.

To see how adding density helps, first disregard spatial localization
and work on a time interval where the cutoffs equal one. Consider
\[
\theta_d=\delta_{n+1}^{16}
\left(\Psi_\mu(\lambda x)-\fint\Psi_\mu\right).
\]
This function has zero mass, $\partial_t\theta_d=0$, and
$\|\theta_d\|_{L^p}\le C\delta_{n+1}^{16}$. At fixed $u^+$,
adding $\theta_d$ to the density changes the continuity defect by
$\theta_du^+$ and the stress by $\theta_du^+\otimes u^+$.
Since $|u^+|\le C\omega$, their contribution to the density
removed at the following stage is small:
\[
\begin{aligned}
&\delta_{n+1}^2|\theta_du^+|
+\delta_{n+1}^4|\theta_du^+\otimes u^+|\\
&\quad\le C\delta_{n+1}^{16}
\bigl(\delta_{n+1}^2\omega+\delta_{n+1}^4\omega^2\bigr)
\Psi_\mu(\lambda x)
\le\tfrac14\delta_{n+1}^{16}\Psi_\mu(\lambda x).
\end{aligned}
\]
Here $\delta_{n+1}^2\omega=\lambda^{-2\gamma+4\gamma/b}\to0$.
The positive density added by $\theta_d$ therefore covers the
subtraction caused by its own errors and leaves room for the other
defects. Stationarity avoids the fast time derivative of a moving
bump; only the slow time cutoff will be differentiated.

We now localize this correction. Choose
\begin{equation}\label{eq:nested-spatial-cutoffs}
\chi_m\in C^\infty_c(B_{5+2m}),\qquad
\chi_m=1\quad\text{on }B_{4+2m},\qquad 0\le m\le8N.
\end{equation}
Thus $\chi_{m+1}=1$ within distance $d$ of $\Supp\chi_m$.
Let $\eta_d$ be the outer time cutoff of Lemma~\ref{lem:cutoffs};
it is one near $\Supp\eta_u$. Define the positive constants
\begin{equation}\label{eq:profile-means}
c_m=\frac{\int\chi_m(x)\Psi_\mu(\lambda x)\dd x}
{\int\chi_m(x)\dd x},
\end{equation}
and set
\begin{equation}\label{eq:stationary-density}
\theta_d(t,x)=\delta_{n+1}^{16}\eta_d(t)
\sum_{m=0}^{8N}\lambda^{-m/2}\chi_m(x)
\bigl(\Psi_\mu(\lambda x)-c_m\bigr).
\end{equation}
Each summand has zero spatial integral. The positive contribution is
$\delta_{n+1}^{16}\eta_d E$, where the smooth, nonnegative,
time-independent weight is
\begin{equation}\label{eq:stationary-weight}
E(x)=\sum_{m=0}^{8N}\lambda^{-m/2}\chi_m(x)\Psi_\mu(\lambda x).
\end{equation}
It is supported in $B_{16N+5}$ and satisfies
$E\ge\Psi_\mu(\lambda\cdot)$ on $B_4$.

The successive cutoffs are needed because $\eta_d'$ produces an
inverse divergence. Already for the first summand,
$\Ro[\chi_0(\Psi_\mu(\lambda\cdot)-c_0)]$ can be nonzero outside
$\Supp\chi_0$. Lemmas~\ref{lem:profile} and \ref{lem:far},
together with the mean estimate in Lemma~\ref{lem:density-basic}
below, give
\[
\begin{aligned}
\bigl|\Ro[\chi_0(\Psi_\mu(\lambda\cdot)-c_0)](x)\bigr|
&\le C\lambda^{-1}d^{-1}\chi_1(x)\Psi_\mu(\lambda x)
+C\lambda^{-1}d^{-1}.
\end{aligned}
\]
The next positive density term has coefficient $\lambda^{-1/2}$,
larger than the $\lambda^{-1}$ in this error. Repeating this
comparison accounts for the sum in $\theta_d$. The final coefficient
is $\lambda^{-4N}$, and
$\lambda^{-4N}\|\Psi_\mu\|_\infty\le C\lambda^{-2N}$.
Lemma~\ref{lem:stationary-corrector} includes the time cutoff in
this estimate.

\subsubsection*{The spatial moment correction}

The stationary correction has zero mass. We now correct the mass of
$\theta_p$ and the first moment of $\theta_p+\theta_d$.
Choose a radial $\phi_0\in C^\infty_c(B(x_0,r_0/2))$ with
$\int\phi_0=1$, and put $\phi_i=-\partial_i\phi_0$ for $i=1,2,3$.
These smooth scalar functions satisfy
\begin{equation}\label{eq:moment-basis}
\begin{aligned}
\int\phi_0&=1,&\int(x-x_0)\phi_0&=0,\\
\int\phi_i&=0,&\int(x-x_0)_j\phi_i&=\delta_{ij}.
\end{aligned}
\end{equation}
Set
\begin{equation}\label{eq:moment-density}
\begin{aligned}
\theta_c(t,x)={}&-\phi_0(x)\int\theta_p(t,y)\dd y\\
&-\sum_{i=1}^3\phi_i(x)\int(y-x_0)_i(\theta_p+\theta_d)(t,y)\dd y.
\end{aligned}
\end{equation}
We finally define
\begin{equation}\label{eq:density}
\theta(t,x)=(\theta_p+\theta_c+\theta_d)(t,x),\qquad
\rho^+=\rho+\theta.
\end{equation}
Products of $x-x_0$ with functions compactly supported in $B$ are
extended periodically by zero when applying the averaging estimates.

\begin{lem}[Density and spatial moments]\label{lem:density-basic}
The functions $\theta_p,\theta_c,\theta_d$ and $\theta$ are smooth
and supported in $\Supp\eta_d\times B_{16N+5}$. They satisfy
\begin{equation}\label{eq:density-moments}
\int\theta_d=0,\qquad \int\theta=0,\qquad
\int(x-x_0)\theta=0.
\end{equation}
Moreover,
\begin{equation}\label{eq:stationary-integral-bounds}
c_m\le C,\qquad \|E\|_{L^p}\le C,\qquad
E\ge\Psi_\mu(\lambda\cdot)\quad\text{on }B_4,
\end{equation}
and
\begin{equation}\label{eq:density-lower-bounds}
\theta_p\ge-\sum_k a_k,
\qquad
\theta_d\ge\delta_{n+1}^{16}\eta_d E-C\delta_{n+1}^{16}.
\end{equation}
With $g=\Psi_\mu-\fint\Psi_\mu$, one also has
\begin{equation}\label{eq:mean-correction-bound}
\left|c_m-\fint\Psi_\mu\right|
=\frac{|\int\chi_mg(\lambda\cdot)|}{\int\chi_m}
\le C\lambda^{-1}d^{-1}.
\end{equation}
\begin{equation}\label{eq:stationary-centered}
\begin{aligned}
\theta_d=\delta_{n+1}^{16}\eta_d\sum_m\lambda^{-m/2}
\Bigl[\chi_mg(\lambda x)
-\left(c_m-\fint\Psi_\mu\right)\chi_m\Bigr].
\end{aligned}
\end{equation}
The stationary corrector also obeys
\begin{equation}\label{eq:stationary-absolute-bound}
|\theta_d|\le C\delta_{n+1}^{16}\eta_d E.
\end{equation}
\end{lem}

\begin{proof}
The support and smoothness follow from
\eqref{eq:principal-density}--\eqref{eq:density},
Proposition~\ref{prop:coefficients}, and the chosen cutoffs.
Equation \eqref{eq:profile-means} gives $\int\theta_d=0$.
Equations \eqref{eq:moment-density} and \eqref{eq:moment-basis} give the remaining identities in
\eqref{eq:density-moments}.

As
$\chi_m=1$ on $B(x_0,r_0)$, Lemma~\ref{lem:profile} gives
$c_m\le\|\Psi_\mu\|_{L^1}/|B(x_0,r_0)|\le C$; since $\Psi_\mu\ge1$, this
gives $|\theta_d|\le C\delta_{n+1}^{16}\eta_d E$. Also
$\|E\|_{L^p}\le\sum_m\lambda^{-m/2}\|\Psi_\mu\|_{L^p}\le C$
by \eqref{eq:stationary-weight} and Lemma~\ref{lem:profile}.
The pointwise comparison on $B_4$ follows from $\chi_0=1$ there.
Finally, $a_k,\Theta_k\ge0$ by Proposition~\ref{prop:coefficients}
and \eqref{eq:moving-profile-bounds}, so
\eqref{eq:principal-density} gives the first lower bound in
\eqref{eq:density-lower-bounds}. The second follows from
\eqref{eq:stationary-density} and $c_m\le C$.

For the mean estimate, apply Lemma~\ref{lem:products} to
$\chi_mg(\lambda\cdot)$ and use
$\int\chi_m\ge|B(x_0,r_0)|$. This proves
\eqref{eq:mean-correction-bound}. Subtracting $\fint\Psi_\mu$ in
each summand of \eqref{eq:stationary-density} gives
\eqref{eq:stationary-centered}.
\end{proof}

\subsection{The continuity defect}\label{sec:new-defects}

We first compute the error left by the transported profiles. By
Lemma~\ref{lem:blocks} and the definition of $f_k$ in
Proposition~\ref{prop:coefficients},
\begin{equation}\label{eq:transport-identity}
\partial_t[a_k(\Theta_k-1)]
+\D[a_k(\Theta_k-1)V_k]=(\Theta_k-1)f_k.
\end{equation}
The sum of these errors, together with the two density corrections,
determines a single scalar potential $P(t,x)$:
\begin{equation}\label{eq:continuity-potential}
\Delta P=\sum_k(\Theta_k-1)f_k
+\partial_t\theta_c+\partial_t\theta_d,
\qquad \fint P=0.
\end{equation}
The right-hand side has integral $\partial_t\int\theta=0$ by
Lemma~\ref{lem:density-basic}. Thus $P$ is well defined and smooth.
Without spatial localization, subtracting $\nabla P$ from the
transported momentum would complete the continuity equation.

For compact support, put $U=B_{16N+8}$ and choose a cutoff radial
about $x_0$,
\begin{equation}\label{eq:outer-cutoff}
\chi\in C^\infty_c(B_{16N+7}),\qquad
\chi=1\quad\text{on }B_{16N+6}.
\end{equation}
Its gradient is separated from all density increments by $2d$, and
$(32N+16)d<(r_+-r_-)/2$ gives
\[
U\subset B\bigl(x_0,\tfrac12(r_-+r_+)\bigr).
\]
Multiplying $\nabla P$ by $\chi$ introduces the divergence error
$\nabla\chi\cdot\nabla P$. We remove it by the Bogovskii operator,
and use a curl to arrange zero angular momentum. Both corrections
are collected in the smooth vector field
\begin{equation}\label{eq:continuity-correction}
\begin{aligned}
R_c(t,x)={}&-\Bog_U[\nabla\chi\cdot\nabla P](t,x)\\
&+\tfrac12\curl_x\left[\phi_0(x)\sum_k
\int a_k(t,y)(\Theta_k(t,y)-1)(y-x_0)\times V_k\dd y\right].
\end{aligned}
\end{equation}
Radial symmetry of $\chi$ and Lemma~\ref{lem:bog} ensure that the
two localization terms
$\chi\nabla P$ and $\Bog_U[\nabla\chi\cdot\nabla P]$ already
have zero angular momentum.

\begin{lem}[The continuity equation]\label{lem:continuity}
For sufficiently large $a>0$, the vector field
\begin{equation}\label{eq:new-R}
R^+=\left(\rho-\sum_k a_k\right)w
+(\theta_c+\theta_d)u^++\chi\nabla P+R_c
\end{equation}
is smooth and supported in $\Supp\eta_d\times U$. It satisfies
\begin{equation}\label{eq:momentum-increment}
\rho^+u^+-R^+-(\rho u-R)
=\sum_k a_k(\Theta_k-1)V_k-\chi\nabla P-R_c,
\end{equation}
and
\[
\partial_t\rho^++\D(\rho^+u^+-R^+)=0.
\]
The modified momentum increment has zero spatial integral and zero
angular momentum:
\begin{equation}\label{eq:momentum-compatibility}
\begin{aligned}
\int\bigl[\rho^+u^+-R^+-(\rho u-R)\bigr]&=0,\\
\int(x-x_0)\times\bigl[\rho^+u^+-R^+-(\rho u-R)\bigr]&=0.
\end{aligned}
\end{equation}
\end{lem}

\begin{proof}
Since $\chi=1$ near the support of $\Delta P$,
\[
\int\nabla\chi\cdot\nabla P=-\int\chi\Delta P=0.
\]
Thus Lemma~\ref{lem:bog} defines the first term of $R_c$, and
\[
\D(\chi\nabla P+R_c)=\Delta P.
\]
By Lemma~\ref{lem:blocks}, $u^+=V_k$ on $\Supp(a_k\Theta_k)$.
Expanding $\rho^+u^+$ and using Proposition~\ref{prop:coefficients}
therefore gives
\[
\rho^+u^+-\rho u
=\sum_k a_k(\Theta_k-1)V_k-R
+\left(\rho-\sum_k a_k\right)w+(\theta_c+\theta_d)u^+.
\]
This proves \eqref{eq:momentum-increment}. Its divergence, the
transport identity \eqref{eq:transport-identity}, and the defining
equation for $P$ give the new continuity equation.

The first spatial moment of $\theta$ vanishes by
Lemma~\ref{lem:density-basic}. Integrating the continuity equation
against $x-x_0$ therefore gives the first identity in
\eqref{eq:momentum-compatibility}. For the second, write
\[
\chi\nabla P=\nabla(\chi P)-P\nabla\chi.
\]
The gradient has zero angular momentum by
Lemma~\ref{lem:moment-identities}, and
$(x-x_0)\times\nabla\chi=0$ by radial symmetry. The Bogovskii
term also has zero angular momentum by Lemma~\ref{lem:bog}.
Finally, Lemma~\ref{lem:moment-identities} and $\int\phi_0=1$
show that the curl in $R_c$ removes exactly the angular momentum
of the transported sum. All fields have the claimed smoothness and
support by construction.
\end{proof}

\subsection{The momentum defect and pressure}

The same transport identity, multiplied by $V_k$, determines the
oscillatory part of the stress. The scalar term $\pi\Id$ in
Proposition~\ref{prop:coefficients} is canceled by $-\pi$ in the
pressure. The time derivative of $\nabla P$ is a gradient, so we set
\begin{equation}\label{eq:new-pressure}
\Pi^+=\Pi-\pi+\chi\partial_tP.
\end{equation}
Only derivatives of the outer cutoff and the compact correction
$R_c$ remain after these cancellations. We collect their symmetric
inverse divergence in $T_c$.

\begin{lem}[The momentum equation]\label{lem:exact}
For sufficiently large $a>0$, the compactly supported symmetric
correction
\begin{equation}\label{eq:stress-correction}
\begin{aligned}
T_c=\Sym_U\Bigl[&\chi\sum_k\fint(\Theta_k-1)f_kV_k
-\Bigl(\sum_k\Rs[(\Theta_k-1)f_kV_k]\Bigr)\nabla\chi\\
&+(\partial_tP)\nabla\chi-\partial_tR_c\Bigr]
\end{aligned}
\end{equation}
is well defined. The symmetric stress
\begin{equation}\label{eq:new-T}
\begin{aligned}
T^+={}&\chi\sum_k\Rs[(\Theta_k-1)f_kV_k]
-(\nabla w+\nabla w^T)\\
&+\left(\rho-\sum_k a_k\right)
(u\otimes w+w\otimes u+w\otimes w)\\
&+(\theta_c+\theta_d)u^+\otimes u^++T_c
\end{aligned}
\end{equation}
completes $Z^+=(\rho^+,u^+,\Pi^+,R^+,T^+)$ to a smooth solution
of \eqref{eq:defect} from Definition~\ref{def:system}. Both $Z^+-Z$
and $(R^+,T^+)$ are supported in $\Supp\eta_d\times U$.
The momentum flux satisfies
\begin{equation}\label{eq:full-flux-increment}
\begin{aligned}
\rho^+u^+\otimes u^+-\rho u\otimes u
={}&\sum_k a_k(\Theta_k-1)V_k\otimes V_k+\pi\Id-T\\
&+\left(\rho-\sum_k a_k\right)
(u\otimes w+w\otimes u+w\otimes w)\\
&+(\theta_c+\theta_d)u^+\otimes u^+.
\end{aligned}
\end{equation}
\end{lem}

\begin{proof}
We verify the cancellation, then the two compatibility conditions
for the compactly supported symmetric inverse divergence.

\proofstep{Step 1: The transported profiles and pressure}
The density decomposition and Lemma~\ref{lem:blocks} give
\[
\rho^+u^+\otimes u^+
=\sum_k a_k\Theta_kV_k\otimes V_k
+\left(\rho-\sum_k a_k+\theta_c+\theta_d\right)u^+\otimes u^+.
\]
Subtracting $\rho u\otimes u$ and using
Proposition~\ref{prop:coefficients} gives
\eqref{eq:full-flux-increment}. Moreover,
\[
\partial_t\sum_k a_k(\Theta_k-1)V_k
+\D\sum_k a_k(\Theta_k-1)V_k\otimes V_k
=\sum_k(\Theta_k-1)f_kV_k.
\]
Insert this identity, the modified momentum increment of
Lemma~\ref{lem:continuity}, and \eqref{eq:new-pressure} into the
momentum equation. Since
$\D(\nabla w+\nabla w^T)=\Delta w$ by
Lemma~\ref{lem:moment-identities}, the terms left to correct are
\[
\begin{aligned}
&\sum_k(\Theta_k-1)f_kV_k
-\D\left(\chi\sum_k\Rs[(\Theta_k-1)f_kV_k]\right)\\
&\qquad-\partial_t(\chi\nabla P+R_c)
+\nabla(\chi\partial_tP).
\end{aligned}
\]
By Lemma~\ref{lem:operators} and the product rule, this is exactly
the vector field in brackets in \eqref{eq:stress-correction}.

\proofstep{Step 2: The compact correction}
The residual just computed is compactly supported in $U$. Before
cancellation it is the time derivative of the modified momentum
increment, plus divergences of compactly supported symmetric matrices
and a compactly supported gradient. Lemmas~\ref{lem:continuity}
and \ref{lem:moment-identities} therefore give zero spatial
integral and zero angular momentum. Lemma~\ref{lem:sym} applies,
so the divergence of $T_c$ is this residual. Adding $T_c$ to the
stress proves the momentum equation. The continuity equation is
Lemma~\ref{lem:continuity}, and $\D u^+=0$ follows from
Lemma~\ref{lem:blocks}. Smoothness and support follow from the
same lemmas and the chosen cutoffs.
\end{proof}

The terms $\theta_du^+$ and $\theta_du^+\otimes u^+$ are the
additional errors created by the density correction. Their effect
on the following density perturbation is smaller than the positive
part of $\theta_d$, as checked above. The potential $\nabla P$
also contains $\Ro\partial_t\theta_d$; its estimate requires the
successive spatial cutoffs.

\section{One step of the iteration: the estimates}\label{sec:estimates}

We prove a lower bound for the total density perturbation $\theta$
that preserves positivity when the construction is repeated. The
new defects determine the possible density subtraction at the next
stage; we compare them with the added density
$\delta_{n+1}^{16}\eta_d E$. We
then prove the integral and derivative estimates needed for the
iteration.

We keep the hypotheses of Proposition~\ref{prop:step} and the
fields of Lemmas~\ref{lem:continuity} and \ref{lem:exact}. The
formulas in these lemmas express the defects in terms of $w$,
the scalar potential $P$, and the compact corrections $R_c,T_c$.
We estimate the corrections first, then the terms entering the
lower bound for $\theta$. Recall
$\delta_{n+1}=\lambda^{-\gamma}$, $\omega=\lambda^{4\gamma/b}$,
$\mu=\lambda^N$, and $\lambda_n=\lambda^{1/b}$.
Since $1000(N^4+20)<\gamma b$, the input bounds in
Proposition~\ref{prop:step} give, for sufficiently large $a>0$,
\begin{equation}\label{eq:old-scale-bounds}
\|(\rho,u,R,T)\|_{C^{20}_{t,x}}+\tau^{-20}+d^{-20}
\le\lambda^{\gamma/1000}.
\end{equation}
Together with Lemma~\ref{lem:cutoffs} and the defining spatial
cutoff bounds, this also controls derivatives of the cutoffs
up to order $20$ by $C\lambda^{\gamma/1000}$. We use this
separation between the old and new scales throughout.

\subsection{Localization and moment errors}

Since $24\gamma<1/4$, errors of size $C\lambda^{-1/2}$ are at most
$\delta_{n+1}^{24}/100$ for sufficiently large $a>0$.
The moment correction $\theta_c$ contains averages against mean-zero
profiles. The corrections $R_c,T_c$ contain such averages or periodic
potentials evaluated on $\Supp\nabla\chi$. This support is separated
from all density increments by $2d$. Spatial averaging and the
far-field estimate therefore make these corrections uniformly small.

\begin{lem}[Localization and spatial moment corrections]\label{lem:small}
For sufficiently large $a>0$, the spatial moment correction satisfies
\begin{equation}\label{eq:small-moments}
\|\theta_c\|_{C^6_{t,x}}\le\lambda^{-3/4}.
\end{equation}
The compact corrections of Lemmas~\ref{lem:continuity} and
\ref{lem:exact} obey
\begin{equation}\label{eq:small-corrections}
\|R_c\|_{C^2_{t,x}}+\|T_c\|_{C^0}\le\lambda^{-1/2}.
\end{equation}
Outside $B_4$, the oscillatory antidivergences satisfy
\begin{equation}\label{eq:small-tail}
\left|\chi\sum_k\Ro[(\Theta_k-1)f_k]\right|
+\left|\chi\sum_k\Rs[(\Theta_k-1)f_kV_k]\right|
\le\lambda^{-3/4}.
\end{equation}
\end{lem}

\begin{proof}
We first estimate the spatial averages and the potentials away from
their sources, then apply the compact inverse divergence operators.

\proofstep{Step 1: The spatial moment correction}
By Lemma~\ref{lem:blocks}, time derivatives of an average
against $\Theta_k-1$ can be moved onto the coefficient. For every
smooth scalar $h$ on $[0,1]\times\TT^3$, integration by parts gives
\begin{equation}\label{eq:averaging-time-derivatives}
\partial_t^j\int h(\Theta_k-1)
=\int[(\partial_t+V_k\cdot\nabla)^jh](\Theta_k-1).
\end{equation}
Apply Lemma~\ref{lem:products} to
$g=\zeta_\mu(\cdot-y_k)-1$, which has $L^1$ norm at most two.
For $h=a_k$ or $h=(x-x_0)_i a_k$ and $0\le j\le6$, combine this with
\eqref{eq:averaging-time-derivatives} and
Proposition~\ref{prop:coefficients}:
\[
\begin{aligned}
\left|\partial_t^j\int h(\Theta_k-1)\right|
&\le C\lambda^{-1}\omega^j\|a_k\|_{C^{j+1}_{t,x}}\\
&\le C\lambda^{-1}\omega^6\delta_{n+1}^{-1}.
\end{aligned}
\]

For the stationary corrector, put $g=\Psi_\mu-\fint\Psi_\mu$.
This profile has zero mean and bounded $L^1$ norm by
Lemma~\ref{lem:profile}. We use the mean estimate and the centered
representation of $\theta_d$ in Lemma~\ref{lem:density-basic}.
Apply Lemma~\ref{lem:products} to $(x-x_0)_i\chi_mg(\lambda\cdot)$
and Lemma~\ref{lem:cutoffs} to $\eta_d$. Since only $\eta_d$
depends on time, for $0\le j\le6$ this gives
\[
\left|\partial_t^j\int(x-x_0)_i\theta_d\right|
\le C\delta_{n+1}^{16}\tau^{-j}\lambda^{-1}d^{-1}.
\]
It follows that the mass and three first-moment coefficients of
$\theta_c$ in Lemma~\ref{lem:density-basic}, together with their
first six time derivatives, are bounded by
\[
C\lambda^{-1}
\left(\omega^6\delta_{n+1}^{-1}+\tau^{-6}d^{-1}\right)
\le C\lambda^{-1+2\gamma},
\]
where we used \eqref{eq:old-scale-bounds} and
$\omega=\lambda^{4\gamma/b}$. The functions $\phi_i$ are fixed,
so
\begin{equation}\label{eq:moment-averaging-bound}
\|\theta_c\|_{C^6_{t,x}}\le C\lambda^{-1+2\gamma}.
\end{equation} For sufficiently large $a>0$, this proves \eqref{eq:small-moments}.

\proofstep{Step 2: The potentials away from their sources}
Fix $0\le i,j\le2$ and
$\mathcal L\in\{\Ro,\Delta^{-1}\Pne\}$. Write
\[
\Theta_k-1=g_k(\lambda x-\lambda\omega t k),
\qquad g_k=\zeta_\mu(\cdot-y_k)-1.
\]
The profile $g_k$ has zero mean and $L^1$ norm at most two. Expand
the time derivatives and apply Lemma~\ref{lem:far} with two inverse
Laplacians to each term. Since
$\Supp f_k\subset\Supp\eta\times B_1$, at distance at least $d$
from $B_1$ we obtain
\begin{equation}\label{eq:oscillatory-farfield}
\begin{aligned}
|\partial_t^j\nabla^i\mathcal L[(\Theta_k-1)f_k]|
&\le Cd^{-8}\lambda^{-4}(\lambda\omega)^2
\|f_k\|_{C^6_{t,x}}\\
&\le Cd^{-8}\lambda^{-2}\omega^3
\|a_k\|_{C^7_{t,x}}\\
&\le C\lambda^{-2+2\gamma}.
\end{aligned}
\end{equation}
The last bound uses Proposition~\ref{prop:coefficients} and
\eqref{eq:old-scale-bounds}. The same argument applies to
$\Rs[(\Theta_k-1)f_kV_k]$, with one additional factor $\omega$.

For the other two terms defining $P$ in Lemma~\ref{lem:continuity},
we use the density corrections. Lemma~\ref{lem:operators} and
\eqref{eq:moment-averaging-bound} give
\[
|\partial_t^j\nabla^i\mathcal L\partial_t\theta_c|
\le C\lambda^{-1+2\gamma}.
\]
For $\theta_d$, put $g=\Psi_\mu-\fint\Psi_\mu$ and use the
centered representation in Lemma~\ref{lem:density-basic}. When
$\dist(x,B_{16N+5})\ge d$, Lemma~\ref{lem:far} gives
\[
|\nabla^i\mathcal L[\chi_mg(\lambda\cdot)](x)|
\le Cd^{-12}\lambda^{-4}.
\]
The mean correction is bounded by $C\lambda^{-1}d^{-3}$ by
Lemmas~\ref{lem:density-basic} and \ref{lem:operators}.
Time derivatives fall only on $\eta_d$; Lemma~\ref{lem:cutoffs}
and \eqref{eq:old-scale-bounds} bound the total contribution by
$C\lambda^{-1+2\gamma}$. Consequently, at distance at least $d$
from $B_{16N+5}$,
\begin{equation}\label{eq:small-farfield}
\begin{aligned}
&|\partial_t^j\nabla^i P|+|\partial_t^j\nabla^i\nabla P|\\
&\qquad+\left|\partial_t^j\nabla^i
\sum_k\Rs[(\Theta_k-1)f_kV_k]\right|\le\lambda^{-3/4}
\qquad(0\le i,j\le2).
\end{aligned}
\end{equation}

\proofstep{Step 3: Averages of the transported profiles}
By Lemma~\ref{lem:products},
\eqref{eq:averaging-time-derivatives}, and
Proposition~\ref{prop:coefficients},
\[
\begin{aligned}
\left|\partial_t^j\sum_k\fint(\Theta_k-1)f_kV_k\right|
&\le C\lambda^{-1}\omega^{j+1}\sum_k\|f_k\|_{C^{j+1}}\\
&\le C\lambda^{-1}\omega^3\delta_{n+1}^{-1}
\qquad(j=0,1).
\end{aligned}
\]
Applying the same averaging identity to
$(x-x_0)\times a_kV_k$ gives
\begin{equation}\label{eq:small-averages}
\left|\partial_t^j\int(x-x_0)\times
\sum_k a_k(\Theta_k-1)V_k\dd x\right|
\le C\lambda^{-1}\omega^3\delta_{n+1}^{-1}
\le\lambda^{-3/4},\qquad 0\le j\le2.
\end{equation}
Thus the coefficients of the curl in $R_c$, together with their
first two time derivatives, are uniformly small.

\proofstep{Step 4: The compact corrections}
The formula for $R_c$ in Lemma~\ref{lem:continuity} contains only
the Bogovskii correction on $\Supp\nabla\chi$ and the curl just
estimated. Lemma~\ref{lem:bog},
\eqref{eq:small-farfield}, and \eqref{eq:small-averages} give
\[
\|R_c\|_{C^2_{t,x}}\le Cd^{-3}\lambda^{-3/4}.
\]
The formula for $T_c$ in Lemma~\ref{lem:exact} contains the
averages from Step 3, the potentials on $\Supp\nabla\chi$,
and $\partial_tR_c$. Lemma~\ref{lem:sym} therefore gives
\[
\|R_c\|_{C^2_{t,x}}+\|T_c\|_{C^0}
\le Cd^{-3}\lambda^{-3/4}\le\lambda^{-1/2},
\]
where $d^{-3}\le\lambda^{3/b}$ by Proposition~\ref{prop:step}.
This proves \eqref{eq:small-corrections}.
Outside $B_4$, the distance to $B_1\supset\Supp f_k$ is at least
$6d$. Estimate \eqref{eq:oscillatory-farfield} and its symmetric
antidivergence counterpart prove \eqref{eq:small-tail}.
\end{proof}

\subsection{The density lower bound}

We prove the lower bound for the full density perturbation required
in Proposition~\ref{prop:step}.

\begin{lem}[Lower bound for the density perturbation]\label{lem:pointwise}
For sufficiently large $a>0$, the total density perturbation
$\theta=\theta_p+\theta_c+\theta_d$ satisfies
\begin{equation}\label{eq:total-density-lower-bound}
\theta\ge-\sum_k a_k
+\delta_{n+1}^2|R^+|+\delta_{n+1}^4|T^+|
-C\delta_{n+1}^{16}.
\end{equation}
In particular, $\theta$ satisfies \eqref{eq:step-positivity} from
Proposition~\ref{prop:step}.
\end{lem}

We first compare the defects from $w$ with the positive density
$\delta_{n+1}^{16}\eta_d E$, then estimate the potential
$\chi\Ro\partial_t\theta_d$. The two products $\theta_du^+$ and
$\theta_du^+\otimes u^+$ are handled in the proof of the lemma.

\begin{lem}[Principal defect terms]\label{lem:principal}
For sufficiently large $a>0$,
\begin{equation}\label{eq:remaining-density-bound}
\left|\rho-\sum_k a_k\right|\le\delta_{n+1}^{-1}.
\end{equation}
The new defects of Lemmas~\ref{lem:continuity} and \ref{lem:exact}
satisfy
\begin{equation}\label{eq:principal-profile-bound}
\begin{aligned}
&\delta_{n+1}^2
\bigl|R^+-\theta_du^+-\chi\Ro\partial_t\theta_d\bigr|
+\delta_{n+1}^4
\bigl|T^+-\theta_du^+\otimes u^+\bigr|\\
&\qquad\le\tfrac14\delta_{n+1}^{16}\eta_d E+\lambda^{-1/2}.
\end{aligned}
\end{equation}
\end{lem}

\begin{proof}
We compare the products containing $w$, the viscous term, and the
oscillatory potentials with $\Psi_\mu(\lambda x)$. The remaining
corrections are uniformly small by Lemma~\ref{lem:small}.

\proofstep{Step 1: Comparison with the stationary profile}
The old-scale bound \eqref{eq:old-scale-bounds} and
Proposition~\ref{prop:coefficients} give
\eqref{eq:remaining-density-bound}.
Lemmas~\ref{lem:blocks} and \ref{lem:profile} imply
\[
\begin{aligned}
|w|&\le C\omega\mathbf 1_{\Supp w},\qquad
|\nabla w|\le C\lambda\mu\omega\mathbf 1_{\Supp w},\\
\mathbf 1_{\Supp w}&\le
C\mu^{-2+\gamma/N}\Psi_\mu(\lambda x).
\end{aligned}
\]
Since $|u|\le\omega$, the products containing $w$ satisfy
\[
\begin{aligned}
&\left|\left(\rho-\sum_k a_k\right)w\right|
+\left|\left(\rho-\sum_k a_k\right)
(u\otimes w+w\otimes u+w\otimes w)\right|\\
&\qquad\le C\delta_{n+1}^{-1}\omega^2
\mu^{-2+\gamma/N}\Psi_\mu(\lambda x).
\end{aligned}
\]
The viscous term is bounded by
\[
|\nabla w+\nabla w^T|
\le C\lambda\mu\omega\mu^{-2+\gamma/N}\Psi_\mu(\lambda x).
\]
For the oscillatory potentials, write $\Theta_k-1=g(\lambda x)$
with $g=\zeta_\mu(\cdot-z)-1$ and
$z=y_k+\lambda\omega t k\in L_k$.
By Proposition~\ref{prop:coefficients},
$\|f_k\|_{C^1}\le C\delta_{n+1}^{-1}$.
Lemma~\ref{lem:profile} gives
\[
|\Ro[(\Theta_k-1)f_k]|+|\Rs[(\Theta_k-1)f_kV_k]|
\le C\lambda^{-1}\omega\delta_{n+1}^{-1}
\mu^{\gamma/N}\Psi_\mu(\lambda x).
\]
The sum of these bounds is a multiple of $\Psi_\mu(\lambda x)$
with coefficient at most
\[
\begin{aligned}
&C\delta_{n+1}^{-1}\bigl(\lambda^{-1}\omega\mu^{\gamma/N}
+\omega^2\mu^{-2+\gamma/N}
+\lambda\mu\omega\mu^{-2+\gamma/N}\bigr)\\
&\qquad\le C\lambda^{-1+3\gamma}
\le\tfrac18\delta_{n+1}^{24}.
\end{aligned}
\]
Here $\mu^{\gamma/N}=\lambda^\gamma$,
$\mu^{-2+\gamma/N}=\lambda^{-2N+\gamma}$, and $27\gamma<1$.

\proofstep{Step 2: Localization and the compact corrections}
All terms estimated in Step 1 are supported in $\Supp\eta_u$,
where $\eta_d=1$ by Lemma~\ref{lem:cutoffs}. On $B_4$,
$\Psi_\mu(\lambda x)\le E(x)$ by
Lemma~\ref{lem:density-basic}. Outside $B_4$, the velocity
perturbation vanishes by Lemma~\ref{lem:blocks}, and the
oscillatory potentials are at most $\lambda^{-3/4}$ by
Lemma~\ref{lem:small}.

It remains to estimate $\theta_cu^+$,
$\theta_cu^+\otimes u^+$, $\chi\Ro\partial_t\theta_c$, and
$R_c,T_c$. Lemmas~\ref{lem:small} and \ref{lem:operators},
together with $|u^+|\le C\omega$, bound their weighted sum by
$C\delta_{n+1}^2\lambda^{-1/2}
+C\lambda^{-3/4}\omega^2$.
Using the formulas for $R^+,T^+$ and the three summands of
$\Delta P$ in Lemma~\ref{lem:continuity} proves the claim for
sufficiently large $a>0$.
\end{proof}

\begin{lem}[Time derivative of the density corrector]\label{lem:stationary-corrector}
For sufficiently large $a>0$,
\begin{equation}\label{eq:stationary-time-bound}
\delta_{n+1}^2|\chi\Ro\partial_t\theta_d|
\le\tfrac18\delta_{n+1}^{16}\eta_d E
+\tfrac1{20}\delta_{n+1}^{24}.
\end{equation}
\end{lem}

\begin{proof}
This is the estimate for which the nested cutoffs in $\theta_d$
were introduced. The inverse divergence of the $m$ th summand is
controlled on the next cutoff, while the final coefficient
$\lambda^{-4N}$ makes its unlocalized bound small. We make this
comparison before using the derivative bound for $\eta_d$.

Let $g=\Psi_\mu-\fint\Psi_\mu$. Differentiating the centered
representation in Lemma~\ref{lem:density-basic} replaces $\eta_d$
by $\eta_d'$. The pointwise estimate of Lemma~\ref{lem:profile} gives
\[
|\Ro[\chi_mg(\lambda\cdot)]|
\le C\lambda^{-1}d^{-1}\Psi_\mu(\lambda x).
\]
At distance at least $d$ from $\Supp\chi_m$,
Lemma~\ref{lem:far} improves this to
$Cd^{-10}\lambda^{-4}\le\lambda^{-3}$. Multiplying by the coefficient of the $m$ th summand, we obtain,
for $m<8N$,
\[
\begin{aligned}
\lambda^{-m/2}|\Ro[\chi_mg(\lambda\cdot)]|
&\le C\lambda^{-1/2}d^{-1}
\lambda^{-(m+1)/2}\chi_{m+1}\Psi_\mu(\lambda x)\\
&\quad+\lambda^{-3-m/2}.
\end{aligned}
\]
The first line is a small multiple of the next positive summand of
$E$. For $m=8N$, the same pointwise estimate and
$\|\Psi_\mu\|_\infty\le C\mu^2$ from Lemma~\ref{lem:profile} give
\[
\lambda^{-4N}|\Ro[\chi_{8N}g(\lambda\cdot)]|
\le C\lambda^{-4N-1}d^{-1}\mu^2.
\]
The mean corrections are bounded by
Lemma~\ref{lem:density-basic}, while $|\Ro\chi_m|\le C$ by
Lemma~\ref{lem:operators}.

The old-scale bound \eqref{eq:old-scale-bounds} gives
$C\delta_{n+1}^2\tau^{-1}d^{-1}\lambda^{-1/2}\le\lambda^{-1/4}$.
Apply Lemma~\ref{lem:cutoffs} and sum the preceding bounds:
\begin{align*}
\delta_{n+1}^2|\chi\Ro\partial_t\theta_d|
&\le\delta_{n+1}^{16}\eta_d^{1-1/(16N)}\lambda^{-1/4}
 \sum_{m<8N}\lambda^{-(m+1)/2}\chi_{m+1}\Psi_\mu(\lambda x)\\
&\quad+C\delta_{n+1}^{16}
 \bigl(\lambda^{-3/4}+\lambda^{-4N-3/4}\mu^2\bigr).
\end{align*}
The second line is at most $\delta_{n+1}^{24}/40$ for sufficiently
large $a>0$, since $\lambda^{-4N}\mu^2=\lambda^{-2N}$ and
$24\gamma<1/4$.
For the first line, Young's inequality gives
\[
\eta_d^{1-1/(16N)}\lambda^{-1/4}
\le\frac{\eta_d}{8}+C_N\lambda^{-4N}.
\]
The weighted sum is bounded by $E$ and $\|E\|_\infty\le C\mu^2$
by Lemma~\ref{lem:profile}. Thus its contribution is at most
\[
\frac18\delta_{n+1}^{16}\eta_d E
+C\delta_{n+1}^{16}\lambda^{-2N}
\le\frac18\delta_{n+1}^{16}\eta_d E
+\frac1{40}\delta_{n+1}^{24}.
\]
This proves the stated bound.
\end{proof}

\begin{proof}[Proof of Lemma~\ref{lem:pointwise}]
By the input bounds in Proposition~\ref{prop:step} and
Lemma~\ref{lem:blocks},
$|u^+|\le|u|+|w|\le C\omega$. Using
Lemma~\ref{lem:density-basic} and
$\delta_{n+1}^2\omega=\lambda^{-2\gamma+4\gamma/b}$, we obtain
\[
\delta_{n+1}^2|\theta_du^+|+\delta_{n+1}^4|\theta_du^+\otimes u^+|
\le C\delta_{n+1}^{16}\eta_d E(\delta_{n+1}^2\omega+\delta_{n+1}^4\omega^2)
\le\tfrac18\delta_{n+1}^{16}\eta_d E.
\]
Combining this bound with Lemmas~\ref{lem:principal} and
\ref{lem:stationary-corrector} gives
\begin{equation}\label{eq:pointwise-defects}
\delta_{n+1}^2|R^+|+\delta_{n+1}^4|T^+|
\le\tfrac12\delta_{n+1}^{16}\eta_d E
+\tfrac1{10}\delta_{n+1}^{24},
\end{equation}
because $C\lambda^{-1/2}\le\delta_{n+1}^{24}/20$.

By Lemmas~\ref{lem:density-basic} and \ref{lem:small}, the full
density perturbation satisfies
\[
\theta\ge-\sum_k a_k+\delta_{n+1}^{16}\eta_d E
-C\delta_{n+1}^{16}-\lambda^{-3/4}.
\]
Comparing this with \eqref{eq:pointwise-defects} proves
\eqref{eq:total-density-lower-bound}, after absorbing the uniformly
small terms into $C\delta_{n+1}^{16}$.
By Proposition~\ref{prop:coefficients} and $\omega=\delta_n^{-4}$,
\[
\sum_ka_k\le48\delta_n^{12}+24\delta_n^4|R|+24\delta_n^8|T|
\le48\delta_n^{12}+\delta_n^2|R|+\delta_n^4|T|
\]
for sufficiently large $a>0$. Since
$48\delta_n^{12}+C\delta_{n+1}^{16}\le\delta_n^2$,
this gives the density lower bound required in
Proposition~\ref{prop:step}.
\end{proof}

\subsection{Integral estimates}

The pointwise bound already controls the defects in $L^p$.
For convergence of the solutions, we also need small density
increments and summable increments of momentum, flux, and pressure.
The latter are estimated directly, since the density and velocity
norms alone do not identify their products.

\begin{lem}[Integral estimates]\label{lem:integral}
For sufficiently large $a>0$, the new defects and the density increment satisfy
\begin{equation}\label{eq:integral-defects}
\|(R^+,T^+)\|_{C_tL^p}\le\delta_{n+1}^8
\end{equation}
and
\begin{equation}\label{eq:integral-density}
\|\rho^+-\rho\|_{C_tL^p}\le\delta_n^2.
\end{equation}
If $n\ge1$, then
\begin{equation}\label{eq:integral-products}
\begin{aligned}
&\|\rho^+u^+-\rho u\|_{C_tL^1}
+\|\rho^+u^+\otimes u^+-\rho u\otimes u\|_{C_tL^1}\\
&\qquad+\|\Pi^+-\Pi\|_{L^1_{t,x}}\le\delta_n^2.
\end{aligned}
\end{equation}
\end{lem}

\begin{proof}
The defect estimates of Lemmas~\ref{lem:principal} and
\ref{lem:stationary-corrector}, together with
Lemma~\ref{lem:density-basic} and $|u^+|\le C\omega$, give
\[
\|R^+\|_{L^p}\le C\delta_{n+1}^{14},\qquad
\|T^+\|_{L^p}\le C\delta_{n+1}^{12}.
\]
For sufficiently large $a>0$, this proves \eqref{eq:integral-defects}. We now estimate
the increments of the density and the nonlinear quantities.

\proofstep{Step 1: The density increment}
For $n\ge1$, Proposition~\ref{prop:coefficients} and the input
bounds in Proposition~\ref{prop:step} give
$\|a_k\|_{L^p}\le\|\sum_ka_k\|_{L^p}\le C\delta_n^{12}$.
For $n=0$, the input is the fixed $Z_*$ and the same proposition
gives $\|a_k\|_\infty\le C_*\omega^{-1}=C_*\delta_0^4$.
In both cases we may use the weaker bound
$\|a_k\|_{L^p}\le C_*\delta_n^4$.

Since $3N(1-1/p)<3N(p-1)<\gamma/b$,
$\|\zeta_\mu\|_{L^p}\le C\mu^{3-3/p}\le C\delta_n^{-1}$.
Lemma~\ref{lem:products} and Proposition~\ref{prop:coefficients}
therefore yield
\[
\begin{aligned}
\|\theta_p\|_{L^p}
&\le C_*\delta_n^{-1}
\bigl(\delta_n^4+\lambda^{-1}\delta_{n+1}^{-1}\omega^{-1}\bigr)
+C_*\delta_n^4\\
&\le C_*\delta_n^3.
\end{aligned}
\]
Lemmas~\ref{lem:small} and \ref{lem:density-basic} give
\[
\|\theta_c\|_\infty\le\lambda^{-3/4},\qquad
\|\theta_d\|_{L^p}\le C\delta_{n+1}^{16}.
\]
For sufficiently large $a>0$, their sum with $\theta_p$ is at most $\delta_n^2$.
In the remaining steps we assume $n\ge1$.

\proofstep{Step 2: The momentum increment}
Use the modified momentum identity in Lemma~\ref{lem:continuity}.
By Lemma~\ref{lem:products} with $q=1$ and the coefficient bound
in Step 1,
\[
\Bigl\|\sum_ka_k(\Theta_k-1)V_k\Bigr\|_{L^1}\le\omega\sum_k\bigl(2\|a_k\|_{L^1}+C\lambda^{-1}\|a_k\|_{C^1}\bigr)
\le C\delta_n^8+C\lambda^{-1}\delta_{n+1}^{-1}.
\]
By Lemma~\ref{lem:products} and
Proposition~\ref{prop:coefficients},
\[
\|\Ro[(\Theta_k-1)f_k]\|_{L^1}
\le C\lambda^{-1}\|f_k\|_{C^1}
\le C\lambda^{-1}\delta_{n+1}^{-1}.
\]
For $\partial_t\theta_d$, use the centered representation and mean
estimate in Lemma~\ref{lem:density-basic}, then apply
Lemma~\ref{lem:products} to obtain
\[
\|\Ro\partial_t\theta_d\|_{L^1}
\le C\tau^{-1}\lambda^{-1}d^{-1}.
\]
The decomposition of $\Delta P$ in Lemma~\ref{lem:continuity}
and the moment bound in Lemma~\ref{lem:small} now give
$\|\chi\nabla P\|_{L^1}\le\lambda^{-1/2}$ for sufficiently large
$a>0$. The same bound holds for $R_c$. Combining these estimates
with \eqref{eq:integral-defects} and the input bounds in
Proposition~\ref{prop:step} yields
\[
\begin{aligned}
\|\rho^+u^+-\rho u\|_{L^1}
&\le\left\|\sum_k a_k(\Theta_k-1)V_k\right\|_{L^1}
+\|\chi\nabla P+R_c\|_{L^1}\\
&\qquad+\|R^+\|_{L^p}+\|R\|_{L^p}
\le C\delta_n^8.
\end{aligned}
\]

\proofstep{Step 3: The momentum flux}
We use the flux identity in Lemma~\ref{lem:exact}. The principal
oscillatory term is bounded by Lemma~\ref{lem:products} and
the coefficient bound in Step 1:

\[
C\omega^2\bigl(\delta_n^{12}
+\lambda^{-1}\delta_{n+1}^{-1}\omega^{-1}\bigr)
\le C\delta_n^4.
\]
The scalar pressure correction has the same size. Indeed,
Proposition~\ref{prop:coefficients} gives
\[
\begin{aligned}
\|\pi\|_{L^1}
&\le C\bigl(\omega\delta_n^8
+\omega\|R\|_{L^1}+\|T\|_{L^1}\bigr)\\
&\le C\delta_n^4,
\end{aligned}
\]
and $\|T\|_{L^1}\le2\delta_n^8$.

For the remaining products, combine
\eqref{eq:remaining-density-bound} with the support and velocity
bounds of Lemma~\ref{lem:blocks}:
\[
\left\|\left(\rho-\sum_k a_k\right)
(u\otimes w+w\otimes u+w\otimes w)\right\|_{L^1}
\le C\delta_{n+1}^{-1}\omega^2\mu^{-3}.
\]
By Lemma~\ref{lem:small} and $|u^+|\le C\omega$,
$\|\theta_cu^+\otimes u^+\|_{L^1}\le\lambda^{-1/2}$. Finally,
Lemma~\ref{lem:density-basic}, and $|u^+|\le C\omega$ give
\[
\begin{aligned}
\|\theta_du^+\otimes u^+\|_{L^1}
&\le C\delta_{n+1}^{16}\omega^2\|E\|_{L^1}\\
&\le C\delta_{n+1}^{12}.
\end{aligned}
\]
Adding these estimates bounds the full flux increment by
$C\delta_n^4$.

\proofstep{Step 4: The pressure increment}
For the pressure in Lemma~\ref{lem:exact}, recall
$\Pi^+-\Pi=-\pi+\chi\partial_tP$ and use
$\|\pi\|_{L^1}\le C\delta_n^4$ from Step 3. By
Lemma~\ref{lem:blocks},
\[
\partial_t\bigl[(\Theta_k-1)f_k\bigr]=-\D\bigl[(\Theta_k-1)f_kV_k\bigr]
+(\Theta_k-1)(\partial_t+V_k\cdot\nabla)f_k .
\]
Apply $\Delta^{-1}\Pne$ to this identity. This operator commutes
with $\D$, and Lemma~\ref{lem:products} applies to
both resulting potentials. Hence
\[
\begin{aligned}
&\bigl\|\partial_t\Delta^{-1}\Pne[(\Theta_k-1)f_k]\bigr\|_{L^1}\\
&\quad\le C\lambda^{-1}
\bigl(\omega\|f_k\|_{C^1}
+\|(\partial_t+V_k\cdot\nabla)f_k\|_{C^1}\bigr)\\
&\quad\le C\lambda^{-1}\omega\delta_{n+1}^{-1}.
\end{aligned}
\]
Further, $\|\Delta^{-1}\Pne\partial_t^2\theta_c\|_{L^1}\le C\lambda^{-3/4}$ by
\eqref{eq:small-moments}. Differentiate the centered representation
in Lemma~\ref{lem:density-basic} twice, then apply
Lemmas~\ref{lem:cutoffs} and \ref{lem:products} to obtain
\[
\|\Delta^{-1}\partial_t^2\theta_d\|_{L^1}
\le C\tau^{-2}\lambda^{-1}d^{-1}.
\]
Hence $\|\Pi^+-\Pi\|_{L^1_{t,x}}\le C\delta_n^4+\lambda^{-1/2}$. For sufficiently
large $a>0$ all three quantities are at most $\frac13\delta_n^2$.
\end{proof}

\subsection{Derivative bounds}

It remains to control the growth of the smooth norms. We use a coarse
polynomial bound: the construction costs at most a power of
$\lambda$ linear in $N$, while the iteration allows $\lambda^{N^2}$.

\begin{lem}[Derivative bounds]\label{lem:derivatives}
For sufficiently large $a>0$,
\[
\|(\rho^+,u^+,R^+,T^+)\|_{C^3_{t,x}}
+\|\Pi^+-\Pi\|_{C^3_{t,x}}\le\tfrac12\lambda^{N^2}.
\]
\end{lem}

\begin{proof}
We first bound the density and velocity perturbations, then the new
errors. The losses from cutoff derivatives are controlled by
\eqref{eq:old-scale-bounds} and absorbed in the extra powers of
$\delta_{n+1}^{-1}$ below.

\proofstep{Step 1: The perturbations and transport errors}
For $0\le j\le6$, all unspecified norms are $C^j_{t,x}$ norms.
Differentiate the moving profiles of Lemma~\ref{lem:blocks} and
use its velocity estimate, together with the derivative bound of
Lemma~\ref{lem:profile}, to obtain
\[
\begin{aligned}
\|\Theta_k\|&\le C\mu^3(\lambda\mu\omega)^j,\\
\|W_k\|&\le C\omega(\lambda\mu\omega)^j,\\
\|\Psi_\mu(\lambda\cdot)\|&\le C\mu^2(\lambda\mu\omega)^j.
\end{aligned}
\]
By Proposition~\ref{prop:coefficients},
\[
\|a_k\|\le\delta_{n+1}^{-1},\qquad
\|\pi\|+\|f_k\|\le C\delta_{n+1}^{-1}\omega.
\]
Apply Lemma~\ref{lem:calculus} to the density perturbations of
Lemma~\ref{lem:density-basic}. Its bound $c_m\le C$ gives
\[
\begin{aligned}
\|\theta_p\|&\le\delta_{n+1}^{-2}\mu^3(\lambda\mu\omega)^j,\\
\|\theta_d\|&\le\delta_{n+1}^{14}\mu^2(\lambda\mu\omega)^j,
\end{aligned}
\]
while $\|\theta_c\|\le1$ by \eqref{eq:small-moments}.
The defining equation for $P$ in Lemma~\ref{lem:continuity} and
the transported sum in Lemma~\ref{lem:exact} involve these same
terms. Thus, for $1\le j\le6$,
\[
\|\Delta P\|_{C^{j-1}_{t,x}}
+\left\|\sum_k(\Theta_k-1)f_kV_k\right\|_{C^{j-1}_{t,x}}
\le\delta_{n+1}^{-2}\mu^3(\lambda\mu\omega)^j.
\]
Here each derivative of a cutoff costs at most a power of
$\tau^{-1}+d^{-1}\le\lambda_n$. The sixth derivative of $\theta_c$
is covered by \eqref{eq:small-moments}.

\proofstep{Step 2: The new defects and pressure}
Apply Lemmas~\ref{lem:operators} and \ref{lem:bog} to $P$ and
the formula for $R_c$ in Lemma~\ref{lem:continuity}. For
$0\le j\le5$, this gives
\[
\|P\|_{C^j_{t,x}}+\|\nabla P\|_{C^j_{t,x}}
+\|R_c\|_{C^j_{t,x}}
\le\delta_{n+1}^{-3}\mu^3(\lambda\mu\omega)^{j+1}.
\]
The formula for $T_c$ in Lemma~\ref{lem:exact} and
Lemma~\ref{lem:sym} then yield
\[
\|T_c\|_{C^3_{t,x}}
\le\delta_{n+1}^{-4}\mu^3(\lambda\mu\omega)^6\omega.
\]

These estimates apply to all fields of
Lemma~\ref{lem:exact}. The products are controlled by
Lemma~\ref{lem:calculus}, and the final symmetric antidivergence by
Lemma~\ref{lem:sym}. Together with the velocity bound, all the
required norms are dominated by
\[
\begin{aligned}
C\delta_{n+1}^{-5}\mu^3\omega^2(\lambda\mu\omega)^9
&\le\lambda^{50N}
\le\lambda^{N^2/2}.
\end{aligned}
\]
Here $50N<N^2/2$, so this is smaller than
$\frac12\lambda^{N^2}$ for sufficiently large $a>0$.
\end{proof}

\subsection{Proof of Proposition~\ref{prop:step}}

\begin{proof}[Proof of Proposition~\ref{prop:step}]
Take the fields $Z^+$ of Lemma~\ref{lem:exact}. That lemma gives
the equations and support, with
$K'=\Supp\eta_d$ and
$U\subset B(x_0,\tfrac12(r_-+ r_+))$.
Lemma~\ref{lem:pointwise} gives the lower bound for the total
density perturbation, and Lemma~\ref{lem:integral} gives the
defect bounds.

The increment bounds follow from \eqref{eq:integral-density},
Lemma~\ref{lem:blocks} and
\[
\|u^+\|_\infty\le2\delta_n^{-1}+C\omega
\le\delta_{n+1}^{-1}.
\]
The nonlinear estimates are \eqref{eq:integral-products}, and the
derivative bound is Lemma~\ref{lem:derivatives}.
\end{proof}

\section{Regularization and proof of Proposition~\ref{prop:main}}\label{sec:induction}

Proposition~\ref{prop:step} controls three derivatives, whereas its
application at the next stage requires twenty. We recover these
derivatives by regularization. The linear form of the approximate
equations lets us mollify without introducing a new equation error;
we then choose the support gaps and close the induction.

\subsection{Regularization}

We mollify differences from the fixed approximate solution in the
linear variables of Definition~\ref{def:system}.

Let $Z_*=(\rho_*,u_*,\bar\Pi,R_*,T_*)$ be as in Lemma~\ref{lem:ext}, with
modified momentum $m_*$ and modified flux $\Phi_*$ as in
Definition~\ref{def:system}. Fix an even, nonnegative scalar
$\kappa\in C^\infty_c(\RR^4)$, supported in the unit ball, with
integral one, and put $\kappa_h(t,x)=h^{-4}\kappa(t/h,x/h)$.
For a function $F$
on $[0,1]\times\TT^3$, smooth and scalar, vector, or matrix valued,
supported in a compact set
$K\subset(0,1)\times\TT^3$ and $0<h<\frac12\dist(K,\{0,1\}\times\TT^3)$,
let $F_h=\kappa_h*F$ (space--time convolution, $F$ extended by zero);
$F_h(t,x)=0$ whenever $\dist((t,x),K)>h$. Elementary properties of convolution give,
for $k\ge0$,
\begin{equation}\label{eq:mollifier}
\|F_h-F\|_{C^k_{t,x}}\le2h\|F\|_{C^{k+1}_{t,x}},\qquad
\|F_h\|_{C^{k}_{t,x}}\le C_kh^{-\max\{k-3,0\}}\|F\|_{C^3_{t,x}} .
\end{equation}

\Needspace{8\baselineskip}
\begin{lem}[Regularization]\label{lem:reg}
Let $Z=(\rho,u,\Pi,R,T)$ be a smooth solution of \eqref{eq:defect} from Definition~\ref{def:system}
such that $Z-Z_*$ is supported in a compact set
$K\subset(0,1)\times\TT^3$, and let $h$ be as above. Define the
regularized density, velocity and pressure by
\begin{equation}\label{eq:regularized-fields}
\begin{aligned}
\rho'&=\rho_*+(\rho-\rho_*)_h,\\
u'&=u_*+(u-u_*)_h,\\
\Pi'&=\bar\Pi+(\Pi-\bar\Pi)_h.
\end{aligned}
\end{equation}
Recover the defects from the regularized modified momentum and flux:
\begin{equation}\label{eq:regularized-defects}
\begin{aligned}
R'&=\rho'u'-m_*-(m-m_*)_h,\\
T'&=\rho'u'\otimes u'-\Phi_*-(\Phi-\Phi_*)_h.
\end{aligned}
\end{equation}
Then $Z'=(\rho',u',\Pi',R',T')$ is a smooth solution of
\eqref{eq:defect} from Definition~\ref{def:system}, and $Z'(t,x)=Z_*(t,x)$ whenever $\dist((t,x),K)>h$.

Set the positive constant
\[
\begin{aligned}
\HZ={}&1+\|(\rho,u,R,T)\|_{C^3_{t,x}}
+\|\Pi-\bar\Pi\|_{C^3_{t,x}}\\
&+\|(\rho_*,u_*,R_*,T_*)\|_{C^3_{t,x}}.
\end{aligned}
\]
The approximation error satisfies
\begin{equation}\label{eq:regularization-error}
\begin{aligned}
&\|Z'-Z\|_{C^0_{t,x}}+\|u'-u\|_{C_tH^s}
+\|\rho'u'-\rho u\|_{C_tL^1}\\
&\qquad+\|\rho'u'\otimes u'-\rho u\otimes u\|_{C_tL^1}
\le C_0h\HZ^3.
\end{aligned}
\end{equation}
The higher derivatives obey
\begin{equation}\label{eq:regularization-derivatives}
\|(\rho',u',R',T')\|_{C^{20}_{t,x}}
\le C_*h^{-17}\HZ^3,
\end{equation}
while the pressure satisfies
\begin{equation}\label{eq:regularization-pressure}
\|\Pi'-\bar\Pi\|_{C^3_{t,x}}
\le\|\Pi-\bar\Pi\|_{C^3_{t,x}}.
\end{equation}
Here $C_0$ depends only on $\kappa$, and $C_*$ also depends on
$\|(\rho_*,u_*,R_*,T_*)\|_{C^{20}_{t,x}}$.
\end{lem}

\begin{proof}
By the linear formulation in Definition~\ref{def:system}, the differences
of $Z$ and $Z_*$ satisfy
\[
\begin{aligned}
\partial_t(\rho-\rho_*)+\D(m-m_*)&=0,\\
\partial_t(m-m_*)+\D(\Phi-\Phi_*)
+\nabla(\Pi-\bar\Pi)-\Delta(u-u_*)&=0.
\end{aligned}
\]
and these differences are supported in $K$. Mollification commutes
with derivatives, so the mollified differences satisfy the same
equations. By \eqref{eq:regularized-fields} and
\eqref{eq:regularized-defects}, $\rho'u'-R'=m_*+(m-m_*)_h$ and
$\rho'u'\otimes u'-T'=\Phi_*+(\Phi-\Phi_*)_h$. Adding the equations for
$Z_*$ shows that $Z'$ solves the system of Definition~\ref{def:system}; also $\D u'=0$ and $T'$
is symmetric. Where $\dist((t,x),K)>h$, all mollified differences vanish, so
$\rho'=\rho_*$, $u'=u_*$, $\Pi'=\bar\Pi$, $R'=R_*$ and $T'=T_*$ there.

\proofstep{Step 1: Approximation in low norms}
In this step $C$ depends only on $\kappa$. We have
$\|\rho-\rho_*\|_{C^3}+\|u-u_*\|_{C^3}+\|\Pi-\bar\Pi\|_{C^3}\le3\HZ$,
$\|m-m_*\|_{C^1}\le C\HZ^2$ and $\|\Phi-\Phi_*\|_{C^1}\le C\HZ^3$. By
\eqref{eq:mollifier},
\[
\|\rho'-\rho\|_{C^0}+\|u'-u\|_{C^2}+\|\Pi'-\Pi\|_{C^0}\le Ch\HZ,
\]
since $\rho'-\rho=(\rho-\rho_*)_h-(\rho-\rho_*)$, and similarly for $u$
and $\Pi$. In particular $\|u'-u\|_{C_tH^s}\le C\|u'-u\|_{C^2}\le Ch\HZ$,
and $\|u'\|_\infty+\|\rho'\|_\infty\le C\HZ$. Next,
\[
R'-R=(\rho'u'-\rho u)-\bigl[(m-m_*)_h-(m-m_*)\bigr],
\]
because $R=\rho u-m_*-(m-m_*)$. Both terms are bounded by $Ch\HZ^2$,
using $\rho'u'-\rho u=(\rho'-\rho)u'+\rho(u'-u)$. In the same way
$\|T'-T\|_{C^0}+\|\rho'u'\otimes u'-\rho u\otimes u\|_{C^0}\le Ch\HZ^3$.
Since $|\TT^3|=1$, the $L^1$ norms are bounded by the $C^0$ norms.

\proofstep{Step 2: Higher derivatives}
By \eqref{eq:mollifier}, for $0\le i\le20$,
\[
\|\rho'\|_{C^i}+\|u'\|_{C^i}
\le C_*+Ch^{-\max\{i-3,0\}}\HZ.
\]
The regularized modified momentum and flux satisfy
\[
\|(m-m_*)_h\|_{C^{20}}+\|(\Phi-\Phi_*)_h\|_{C^{20}}
\le Ch^{-17}\HZ^3.
\]
In differentiating $\rho'u'$ or $\rho'u'\otimes u'$, the twenty
derivatives are distributed among the factors. The powers of $h^{-1}$
in any product add up to at most seventeen, since
\[
\sum_l\max\{i_l-3,0\}\le17
\qquad\text{whenever}\qquad \sum_l i_l=20.
\]
Leibniz' rule therefore gives the bound for $R'$ and $T'$ in
\eqref{eq:regularization-derivatives}. Finally,
$\Pi'-\bar\Pi=(\Pi-\bar\Pi)_h$, and convolution with a probability
density does not increase the $C^3$ norm. This proves
\eqref{eq:regularization-pressure}.

\end{proof}

\subsection{Proof of Proposition~\ref{prop:main}}

\begin{proof}
At each stage we first regularize the difference from $Z_*$, then
apply Proposition~\ref{prop:step}. The regularization restores the
derivative bounds required for the next perturbation. We choose its
scale so that its effect on the low norms and on the pointwise density
bound is much smaller than the allowed increment.

\proofstep{Step 1: Choice of the support margins}
Write $J=(t_-,t_+)$, and choose $\ell>0$ with
$K_*\subset(t_-+2\ell,t_+-2\ell)\times B$. Define
\[
I_n=\bigl[t_-+\ell(1+2^{-n}),\,t_+-\ell(1+2^{-n})\bigr]
\]
and
\[
I_n'=\bigl[t_-+\ell(1+\tfrac342^{-n}),\,
t_+-\ell(1+\tfrac342^{-n})\bigr].
\]
Then $I_n\subset I_n'\subset\operatorname{int}I_{n+1}$ and
$K_*\subset I_0\times B$. Both the inner time gap and
$\tau_n=\dist(I_n',\partial I_{n+1})$ equal $\ell2^{-n}/4$.

For the spatial supports, let
\[
r_\infty=\tfrac12(r_0+r_B),\qquad
r_n=r_\infty-(r_\infty-r_0)2^{-n},
\]
and set
\[
r_0^-=r_0,\qquad
r_n^-=\tfrac12(r_{n-1}+r_n)\quad(n\ge1).
\]
All supports will lie in the compact subset
\[
K_\infty=[t_-+\ell,t_+-\ell]\times\overline B(x_0,r_\infty)
\subset Q.
\]
We use the spatial gap and regularization scale
\[
d_n=\frac{r_{n+1}-r_n}{64N+40},\qquad
h_n=\lambda_n^{-N^3}.
\]
Since inverse support gaps grow at most like $C2^n$, whereas
$\lambda_n=a^{b^n}$, for sufficiently large $a>0$ we have uniformly in $n$
\begin{equation}\label{eq:gaps}
\begin{aligned}
\tau_n^{-1}+d_n^{-1}&\le\lambda_n,\\
h_n&<\min\{\tau_n,\ell/2\}\qquad(n\ge0),\\
h_n&<\tfrac14(r_n-r_{n-1})\qquad(n\ge1).
\end{aligned}
\end{equation}

\proofstep{Step 2: The bounds to be propagated}
We construct $Z_n$, starting from $Z_0=Z_*$, so that $Z_n-Z_*$ and
$(R_n,T_n)$ are supported in
$I_n\times\overline B(x_0,r_n^-)$. The norm bounds are
\begin{equation}\label{eq:induction-bounds}
\begin{aligned}
\|(\rho_n,u_n,R_n,T_n)\|_{C^3_{t,x}}
+\|\Pi_n-\bar\Pi\|_{C^3_{t,x}}&\le\lambda_n^{N^2},\\
\|u_n\|_\infty&\le\delta_n^{-1},\\
\|(R_n,T_n)\|_{C_tL^p}&\le\delta_n^8\qquad(n\ge1).
\end{aligned}
\end{equation}
For the density, we require the direct pointwise lower bound
\begin{equation}\label{eq:induction-positivity}
\rho_n\ge\rho_*-\delta_0
+\delta_n^2|R_n|+\delta_n^4|T_n|
+2\sum_{j\ge n}\delta_j^2.
\end{equation}
The defect terms cover the density removed by the next principal
perturbation; the tail covers the remaining uniform losses. The
initial conditions hold for sufficiently large $a>0$, since $Z_*$
is fixed and
\[
\delta_0^2\|R_*\|_\infty+\delta_0^4\|T_*\|_\infty
+2\sum_{j\ge0}\delta_j^2\le\delta_0.
\]

\proofstep{Step 3: Regularization before the perturbation}
Assume the preceding bounds at level $n$. If $n=0$, put
$Z_0'=Z_0$. If $n\ge1$, apply Lemma~\ref{lem:reg} with
\[
K=I_n\times\overline B(x_0,r_n^-),\qquad h=h_n,
\]
and denote the result by $Z_n'$. By
\eqref{eq:induction-bounds}, the quantity $\HZ$ in that lemma is at
most $3\lambda_n^{N^2}$. Estimate \eqref{eq:regularization-error} gives
\begin{equation}\label{eq:induction-regularization-error}
\begin{aligned}
&\|Z_n'-Z_n\|_{C^0_{t,x}}
+\|u_n'-u_n\|_{C_tH^s}\\
&\quad+\|\rho_n'u_n'-\rho_nu_n\|_{C_tL^1}\\
&\quad+\|\rho_n'u_n'\otimes u_n'
-\rho_nu_n\otimes u_n\|_{C_tL^1}\\
&\qquad\le27C_0\lambda_n^{-N^3+3N^2}\le\lambda_n^{-N^3/2}\le\delta_n^{16}.
\end{aligned}
\end{equation}
By \eqref{eq:regularization-derivatives},
\[
\begin{aligned}
\|(\rho_n',u_n',R_n',T_n')\|_{C^{20}_{t,x}}
&\le27C_*\lambda_n^{17N^3+3N^2}\\
&\le\lambda_n^{N^4/2}\le\lambda_n^{N^4}.
\end{aligned}
\]

By \eqref{eq:gaps}, the regularized support is contained in
$I_n'\times\overline B(x_0,r_n)$. The same set contains $K_*$, so
it also contains the supports of $R_n',T_n'$. Combining \eqref{eq:induction-regularization-error} with
\eqref{eq:induction-bounds} gives
\[
\|u_n'\|_\infty\le2\delta_n^{-1},\qquad
\|(R_n',T_n')\|_{C_tL^p}\le2\delta_n^8
\quad(n\ge1),
\]
and
\begin{equation}\label{eq:regularized-density}
\rho_n'\ge\rho_*-\delta_0
+\delta_n^2|R_n'|+\delta_n^4|T_n'|
+2\sum_{j\ge n}\delta_j^2-3\delta_n^{16}.
\end{equation}
Consequently, $Z_n'$ satisfies all the hypotheses for one step,
with
\[
\begin{aligned}
I'&=I_n',&J'&=\operatorname{int}I_{n+1},\\
r_-&=r_n,&r_+&=r_{n+1},\\
\tau&=\tau_n,&d&=d_n.
\end{aligned}
\]

\proofstep{Step 4: The next perturbation}
Let $Z_{n+1}=(Z_n')^+$ be given by Proposition~\ref{prop:step}.
The support conclusion places $Z_{n+1}-Z_n'$ and its defects
inside $I_{n+1}\times\overline B(x_0,r_{n+1}^-)$.
The difference $Z_n'-Z_*$ is supported in the same set, so the
required support property holds at level $n+1$.

By Proposition~\ref{prop:step} and
\eqref{eq:regularization-pressure},
\[
\begin{aligned}
&\|(\rho_{n+1},u_{n+1},R_{n+1},T_{n+1})\|_{C^3_{t,x}}
+\|\Pi_{n+1}-\bar\Pi\|_{C^3_{t,x}}\\
&\qquad\le\tfrac12\lambda_{n+1}^{N^2}
+\|\Pi_n'-\bar\Pi\|_{C^3_{t,x}}\\
&\qquad\le\tfrac12\lambda_{n+1}^{N^2}+\lambda_n^{N^2}
\le\lambda_{n+1}^{N^2}.
\end{aligned}
\]
The remaining norm bounds in \eqref{eq:induction-bounds} follow
from Proposition~\ref{prop:step}. For the density, apply its lower
bound for the full perturbation to the regularized fields and use
\eqref{eq:regularized-density}:
\[
\begin{aligned}
\rho_{n+1}\ge{}&\rho_*-\delta_0
+\delta_{n+1}^2|R_{n+1}|+\delta_{n+1}^4|T_{n+1}|\\
&+2\sum_{j\ge n}\delta_j^2-\delta_n^2-3\delta_n^{16}.
\end{aligned}
\]
Since $3\delta_n^{16}\le\delta_n^2$, the last line is at least
$2\sum_{j\ge n+1}\delta_j^2$. This proves
\eqref{eq:induction-positivity} at level $n+1$ and closes the
induction.

\proofstep{Step 5: Convergence bounds and the density lower bound}
The common support property follows from the construction of
$K_\infty$. Combining Proposition~\ref{prop:step} and
\eqref{eq:induction-regularization-error} gives
\[
\begin{aligned}
\|\rho_{n+1}-\rho_n\|_{C_tL^p}
+\|u_{n+1}-u_n\|_{C_tH^s}
&\le2\delta_n^{16}+\delta_n^2+\delta_{n+1}^{16}\\
&\le\delta_n.
\end{aligned}
\]
This proves the density and velocity increment bound. The defect
bound \eqref{eq:iteration-defects} is the last estimate in
\eqref{eq:induction-bounds}.

For $n\ge1$, Proposition~\ref{prop:step} and
\eqref{eq:induction-regularization-error} bound the total increment
of momentum, flux and pressure by
$\delta_n^2+3\delta_n^{16}\le\delta_n$. For the regularized pressure
we also use
$\|\Pi_n'-\Pi_n\|_{L^1}\le\|Z_n'-Z_n\|_{C^0}$.
This proves the required momentum, flux, and pressure increment
bound.

Finally, dropping the nonnegative defect and tail terms in
\eqref{eq:induction-positivity} gives
$\rho_n\ge\rho_*-\delta_0$. This is the required density lower
bound and completes the proof.
\end{proof}

\appendix

\section{Elementary auxiliary estimates}\label{app:elementary}

We prove Lemmas~\ref{lem:cutoffs}, \ref{lem:calculus}, and
\ref{lem:operators}. The cutoff construction gives the nested supports
and the derivative bound used for positivity; the kernel calculation
gives both the periodic scaling and the endpoint integral estimates.

\subsection{Time cutoffs}

\begin{proof}[Proof of Lemma~\ref{lem:cutoffs}]
We first construct a smooth increasing cutoff on the real line, then
rescale it at the two endpoints of $I$.

\proofstep{Step 1: A cutoff with a controlled derivative}
Define $\eta_0\in C^\infty(\RR;[0,1])$ by
\[
\eta_0(t)=
\begin{cases}
0,&t\le0,\\[2pt]
\displaystyle\frac{e^{-1/t}}{e^{-1/t}+e^{-1/(1-t)}},&0<t<1,\\[6pt]
1,&t\ge1.
\end{cases}
\]
On $(0,1)$, differentiation gives
\[
\eta_0'=\eta_0(1-\eta_0)
\bigl(t^{-2}+(1-t)^{-2}\bigr)\ge0.
\]
For $0<t\le1/2$, $\eta_0(t)\le e^2e^{-1/t}$, so
\[
|\eta_0'|\eta_0^{-1+1/(16N)}
\le2t^{-2}\eta_0^{1/(16N)}
\le2e^{1/(8N)}t^{-2}e^{-1/(16Nt)}\le C_N.
\]
For $t\ge1/2$, $\eta_0\ge1/2$ and $\eta_0'$ is bounded. Thus
\begin{equation}\label{eq:basic-time-cutoff}
|\eta_0'|\le C_N\eta_0^{1-1/(16N)}\quad\text{on }\RR.
\end{equation}

\proofstep{Step 2: The three nested cutoffs}
Set
\[
\eta(t)=\eta_0\Bigl(\frac{10(t-t_-)}\tau+2\Bigr)
\eta_0\Bigl(\frac{10(t_+-t)}\tau+2\Bigr),
\]
and define $\eta_u$ and $\eta_d$ with $2$ replaced by $4$ and $6$,
respectively. Then $\eta=1$ on
$[t_--\tau/10,t_++\tau/10]$, and
\[
\begin{aligned}
\Supp\eta&\subset[t_--\tau/5,t_++\tau/5]\subset\{\eta_u=1\},\\
\Supp\eta_u&\subset[t_--2\tau/5,t_++2\tau/5]\subset\{\eta_d=1\},\\
\Supp\eta_d&\subset[t_--3\tau/5,t_++3\tau/5]\Subset J'.
\end{aligned}
\]
The gaps in the definitions make each cutoff one near the preceding
support.

The desired $C^j$ bounds follow by scaling,
since $\tau<1$. Apply \eqref{eq:basic-time-cutoff} to each factor of
$\eta_d$. Both factors lie in $[0,1]$, so each term in the product
rule is bounded by $C_N\tau^{-1}\eta_d^{1-1/(16N)}$. This gives the
desired bound for $\eta_d'$.
\end{proof}

\subsection{Calculus inequalities}

\begin{proof}[Proof of Lemma~\ref{lem:calculus}]
Leibniz' rule gives the product estimate, since the binomial
coefficients for a derivative of order at most $j$ sum to at most
$2^j$.

For the composition estimate,
$\|\Gamma\circ v\|_\infty\le\|\Gamma\|_{C^0(\mathcal K)}$.

For $1\le i\le j$, the Fa\`a di Bruno formula writes each derivative
of order $i$ as a sum of at most $C(j,m)$ terms of the form
\[
(\nabla^l\Gamma)(v)
\bigl[\partial^{\alpha_1}v,\dots,\partial^{\alpha_l}v\bigr],
\]
where the multi-indices satisfy
\[
|\alpha_a|\ge1,\qquad \sum_{a=1}^l|\alpha_a|=i.
\]
Here $1\le l\le i$, and the derivatives are in $(t,x)$. Every such
term is bounded by
\[
\|\nabla^l\Gamma\|_{C^0(\mathcal K)}
\|v\|_{C^j_{t,x}}^l
\le\|\Gamma\|_{C^j(\mathcal K)}
\bigl(1+\|v\|_{C^j_{t,x}}\bigr)^j.
\]
This proves the composition bound. For positive-order derivatives,
only derivatives of $\Gamma$ of orders $1,\dots,j$ occur. Replacing
their norms on $\mathcal K$ by their global bounds proves the last
assertion.
\end{proof}

\subsection{Periodic kernels}

\begin{proof}[Proof of Lemma~\ref{lem:operators}]
We identify the singularity of each inverse Laplacian, then deduce
the operator bounds and the scaling identities.

\proofstep{Step 1: The periodic Green functions}
The mean-zero periodic Green function has Fourier coefficients
\[
\widehat G(0)=0,\qquad
\widehat G(k)=-\frac1{4\pi^2|k|^2}
\quad(k\in\ZZ^3\setminus\{0\}).
\]
Thus $G_2=G*G$ is mean-zero, $\Delta G_2=G$, and
$\Delta^{-2}\Pne f=G_2*f$.

Choose a scalar $\chi\in C^\infty_c(B(0,1/2))$ equal to one on
$B(0,1/4)$, and regard it as periodic on the fundamental cube.
Since $\Delta(-1/(4\pi|z|))=\delta_0$, the distribution
\[
G(z)+\frac{\chi(z)}{4\pi|z|}
\]
has a smooth Laplacian and hence is smooth by elliptic regularity.
Similarly, $\Delta|z|=2/|z|$ gives
\[
G_2(z)+\frac{\chi(z)|z|}{8\pi}\in C^\infty(\TT^3).
\]
Both Green functions are smooth away from the origin. Differentiating
these expressions proves the asserted kernel bounds.

\proofstep{Step 2: Kernel bounds and boundedness}
By definition, the kernel of $\Ro$ is
$\nabla G$; the entries of the kernel of
$\Rs$ are linear combinations of entries of $\nabla G$ and
$\nabla^3G_2$. The kernel of $\Delta^{-1}\Pne$ is $G$.
The Green-function bound just proved gives the asserted kernel
estimates for these operators.

All these kernels are integrable and have zero mean. Young's
inequality gives boundedness on $L^q$ for $1\le q\le\infty$.
Convolution commutes with translations and derivatives, including
time derivatives, so the same bounds hold on $C^j_{t,x}$.

\proofstep{Step 3: Scaling and integrable derivatives}
The nonzero Fourier multipliers of $\Ro$ and $\Rs$ are
homogeneous of degree $-1$; the multiplier of
$\Delta^{-1}\Pne$ has degree $-2$. For integer $\lambda$,
\[
\widehat{g(\lambda\cdot)}(\lambda k)=\widehat g(k),
\qquad
\widehat{g(\lambda\cdot)}(k)=0
\quad\text{if }k\notin\lambda\ZZ^3.
\]
Applying the multipliers at these frequencies proves the scaling
identities.

Finally, the kernel of $\Delta^{-2}\partial^\alpha\Pne$ is
$\partial^\alpha G_2$. Its singularity has order $1-|\alpha|$,
which is integrable for $|\alpha|\le3$.
Another application of Young's inequality gives the last assertion.
\end{proof}

\vspace{1.5em}
\noindent{\sc Xingyu Li}\\
Department of Mathematics and Geosciences, University of Trieste,
via Valerio 12/1, Trieste, 34127 Italy\\
{\it Email address:} \texttt{xingyuli92@gmail.com}

\end{document}